\documentclass[12pt]{amsart}
\pdfoutput=1
\usepackage{amssymb, amscd,  stmaryrd}
\usepackage{enumerate}
\usepackage{graphicx}
\usepackage[usenames,dvipsnames]{color}
\usepackage{color}

\newcommand{\pa}{\partial}

\newcommand{\maC}{\mathcal C}

\newcommand{\maH}{\mathcal H}

\newcommand{\maP}{\mathcal P}

\newcommand{\maO}{\mathcal O}

\newcommand{\maV}{\mathcal V}

\newcommand{\maT}{\mathcal T}

\newcommand{\beeta}{\boldsymbol {\eta}}
\newcommand{\bmu}{\boldsymbol {\mu}}

\newcommand{\bv}{\textit{\textbf v}}
\newcommand{\bu}{\textit{\textbf u}}
\newcommand{\ba}{\textit{\textbf a}}

\newcommand{\bbf}{\textit{\textbf f}}
\newcommand{\bbg}{\textit{\textbf g}}
\newcommand{\bbh}{\textit{\textbf h}}

\newcommand\comment[1]{}

\newcommand{\maK}{\mathcal K}

\newtheorem{theorem}{Theorem}[section]
\newtheorem{proposition}[theorem]{Proposition}

\newtheorem{lemma}[theorem]{Lemma}

\theoremstyle{definition}
\newtheorem{definition}[theorem]{Definition}
\theoremstyle{definition}
\newtheorem{algorithm}[theorem]{Algorithm}
\theoremstyle{remark}
\newtheorem{remark}[theorem]{Remark}

\author[Y.-J. Lee]{Young-Ju Lee} \address{Young-Ju Lee, Department of Mathematics, Rutgers, The State University of New Jersey, Piscataway, NJ 08854}
\email{leeyoung@math.rutgers.edu}

\author[H. Li]{Hengguang Li} \address{Hengguang Li, Department of Mathematics,
Wayne State University, Detroit, MI 48202}
\email{hli@math.wayne.edu}

\thanks{H. Li was supported in part by the NSF grant DMS-1115714 and Y.-J. Lee was partially supported by the NSF grant DMS-0915028.}

\begin{document}
\title[axisymmetric Stokes equations]{Axisymmetric Stokes equations in polygonal domains: regularity and finite element approximations}

\date{\today}

\begin{abstract} 
We study the regularity and finite element approximation of the axisymmetric Stokes problem on a polygonal domain $\Omega$. In particular, taking into account the singular coefficients in the equation and non-smoothness of the domain, we establish the well-posedness and full regularity of the solution in new weighted Sobolev spaces $\maK^m_{\mu, 1}(\Omega)$. Using our a priori results, we give a specific construction of graded meshes on which the Taylor-Hood mixed method approximates singular solutions at the optimal convergence rate. Numerical tests are presented to confirm the theoretical results in the paper.

\end{abstract}
\maketitle
\section{Introduction}

The finite element  simulation of  partial differential equations in 3D usually presents a serious computational challenge, due to  the high-dimensional nature of the problem. In particular, the computational complexity is even higher when high-order discretization schemes are applied to systems of equations. For axisymmetric problems, in order to improve the effectiveness of the numerical algorithm, a highly effective technique is to reduce the dimension of the computational domain using properties of axisymmetry.

Consider the 3D Stokes equations in a bounded domain. When both the data and domain are invariant with respect to the rotation about the $z$-axis, the 3D Stokes problem can be reduced into two decoupled 2D equations: a vector saddle point problem (the axisymmetric Stokes equations) and a scalar elliptic problem (the azimuthal Stokes equation). Despite the potential of substantial savings in computations, this process leads to irregular equations with singular coefficients, which together with the non-smoothness of the domain, raises the difficulty in analyzing the problem on both the continuous and discrete levels. 
In this paper, we shall study the well-posedness, regularity, and optimal finite element approximations of the axisymmetric  Stokes problem with singular solutions.

The numerical approximation of axisymmetric problems has been of great interest in recent years. A comprehensive discussion on spectral methods for different axisymmetric problems and on corresponding weighted Sobolev spaces can be found in \cite{book}.  Assuming the \textit{full regularity} in weighted spaces, we also mention that finite element/multigrid methods for the axisymmetric Laplace operator were formulated in \cite{GP06,MR82}; the partial Fourier approximation of axisymmetric linear elasticity problems were treated in \cite{MR1706008};  for the theoretical justification and numerical approximation of the axisymmetric Maxwell equations, we refer the readers to \cite{ACL02,CGP08} and references therein. In particular, for axisymmetric Stokes equations, Belhachmi, Bernardi, and  Deparis \cite{BBD06} established the stability and approximation properties for the P1isoP2/P1 mixed method, while Lee and Li  \cite{ LL11} proved that the general Taylor-Hood mixed methods are stable.  Several stability results and local interpolation operators  will be borrowed from these works for the analysis in this paper.

Although there is extensive literature in developing optimal finite element methods for elliptic equations with singular solutions, there are few works on the finite element  treatment for singular solutions  of axisymmetric equations, most of which are for the axisymmetric Poisson equation. For example, see \cite{MR1411853, Li11, Nkemzi05}.

Compared with standard elliptic problems, the main difficulties in numerical analysis of  singular solutions of axisymmetric equations arise in handling both continuous and discrete equations. Namely,  on the continuous level, it  requires a good understanding on the singular solution in the original 3D problem from the non-smoothness of the domain (e.g., conical points and edges) and on the interaction  between the axisymmetric equations and the 3D problem. The establishment of isomorphic mappings in special weighted spaces is critical. On the discrete level, because of the singular coefficients and vanishing weights in the function space, the approximation properties of polynomials and  the stability of certain operators to the finite element space have to be reconsidered in the weighted sense. 

As mentioned above, we shall focus on the a priori estimates and the finite element approximation of the axisymmetric Stokes problem, especially when the solution has singularities due to the singular coefficients and the  non-smooth domain. In particular, we shall introduce new weighted Sobolev spaces (Definition \ref{def.regular1}) and establish  the full regularity up to any order  in these spaces (Theorem \ref{thm.stokesreg}). Then, we apply our regularity result to the Taylor-Hood mixed method for the axisymmetric Stokes problem. Using local estimates on special interpolation operators in weighted spaces, we give a construction of  a sequence of graded meshes, on which the mixed finite element approximation converges to the singular solution at the optimal rate (Theorem \ref{thm.main2}),  as is achieved in the finite element method for smooth solutions of elliptic equations \cite{BS02, Ciarlet78}. Note that  the isomorphic mappings (Proposition \ref{prop.1}) are only for the usual Sobolev space. Therefore, the existing 3D regularity results in weighted spaces of Kondrat$^\prime$ev's type can not be directly translated to the new weighted space.

To the best of our knowledge, this is the first  full regularity result in weighted Sobolev spaces for axisymmetric Stokes equations. It is expected that our theory can provide guidelines on the regularity estimates  for other axisymmetric problems involving vector fields. Although our theory is applied to the Taylor-Hood finite element methods in this paper,  the approach  applies to other stable mixed methods for the axisymmetric Stokes problem, in which the local approximation depends on the local patch in the triangulation. The regularity result will also be useful for analysis of many other aspects of the finite element method. 

The rest of the paper is organized as follows. In Section \ref{sec2}, we describe the axisymmetric Stokes problem and its mixed weak formulation. In addition, we introduce two types of weighted Sobolev spaces (Definitions \ref{def.regular} and \ref{def.regular1}) to carry out the analysis. Useful connections between these weighted spaces are also discussed. In Section \ref{sec3}, using local estimates for different parts of the domain and certain isometric mappings, we provide our first main result in Theorem \ref{thm.stokesreg}, the full regularity estimates in weighted spaces for axisymmetric Stokes equations. The solution is shown to be always smoother than the given data in weighted spaces although there may be singularities in the solution. In Section \ref{sec4}, we propose a construction of a sequence of graded meshes for singular solutions. Based on the regularity results in Section \ref{sec3}, we give a specific range for the grading parameter $\kappa$, such that the Taylor-Hood mixed method approximates singular solutions at the optimal rate. This is our second main result, which is formulated in Theorem \ref{thm.main2}. In section \ref{sec5}, we provide numerical results on graded meshes for different singular solutions. These tests convincingly verify our theoretical prediction on the convergence rates and on the construction of optimal graded meshes for singular solutions of the axisymmetric Stokes problem.\\\\
\textbf{Acknowledgements.} We would like to thank Douglas N. Arnold and Victor Nistor for useful discussions. Special thanks go to Serge Nicaise for pointing out critical references for this research.

\section{Preliminaries and notation} \label{sec2}

\subsection{Axisymmetric Stokes equations and function spaces} 

Let $\tilde\Omega\subset \mathbb R^3$ be a 3D domain obtained by the rotation of a 2D polygonal (meridian) domain $\Omega\subset\mathbb R^2$  in the $rz$-plane about the $z$-axis, where $r=\sqrt{x^2+y^2}$ is the distance to the $z$-axis. Namely, $\tilde \Omega:=\Omega\times [0, 2\pi)$. (See Figure \ref{fig.domain} for example.)  A 3D vector field $\tilde\bv=(v_1, v_2, v_3)$ (resp. function $\tilde v$) is axisymmetric if 
\begin{eqnarray}\label{eqn.1}
\mathcal R_{-\sigma} (\tilde{\bv}\circ \mathcal R_\sigma)=\tilde{\bv}  \qquad ({\rm{resp.}}\   [ \tilde v\circ \mathcal R_\sigma](x, y, z)=\tilde v(x, y, z)),\quad\forall\ \sigma\in[0, 2\pi),
   \end{eqnarray}
where $\mathcal R_\sigma$ is the rotation around the $z$-axis with angle $\sigma$.  In addition, the vector field can also be expressed by its radial, angular, and axial components 
\begin{eqnarray*}\label{eqn.3}
\tilde\bv=(v_r, v_\theta, v_z)=(v_1\cos\theta +v_2\sin\theta , - v_1\sin\theta+v_2\cos \theta , v_3).
\end{eqnarray*} 

\begin{figure}
\includegraphics[scale=0.27]{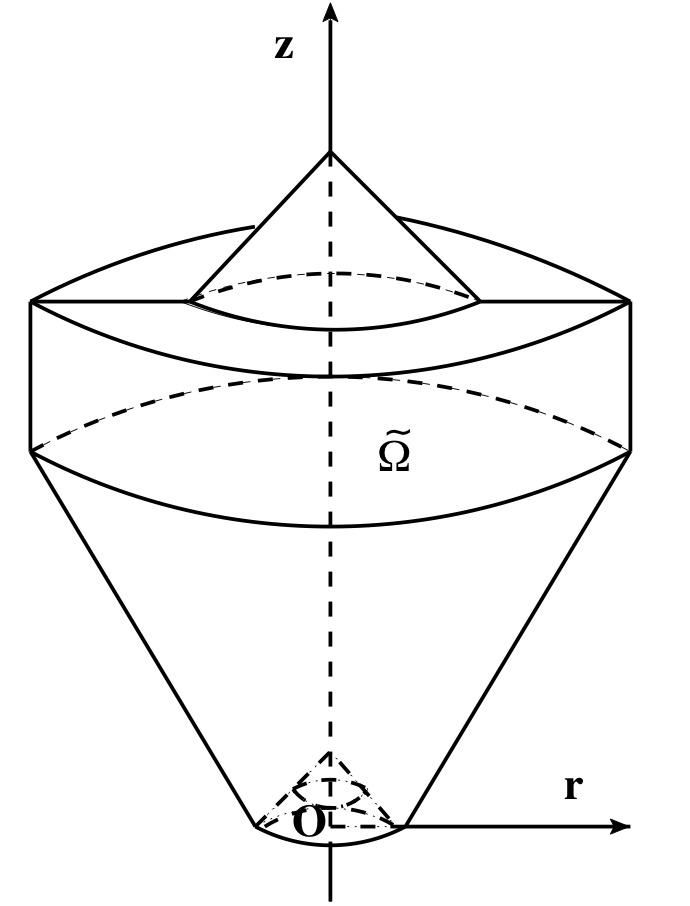}\hspace{3.5cm}\includegraphics[scale=0.27]{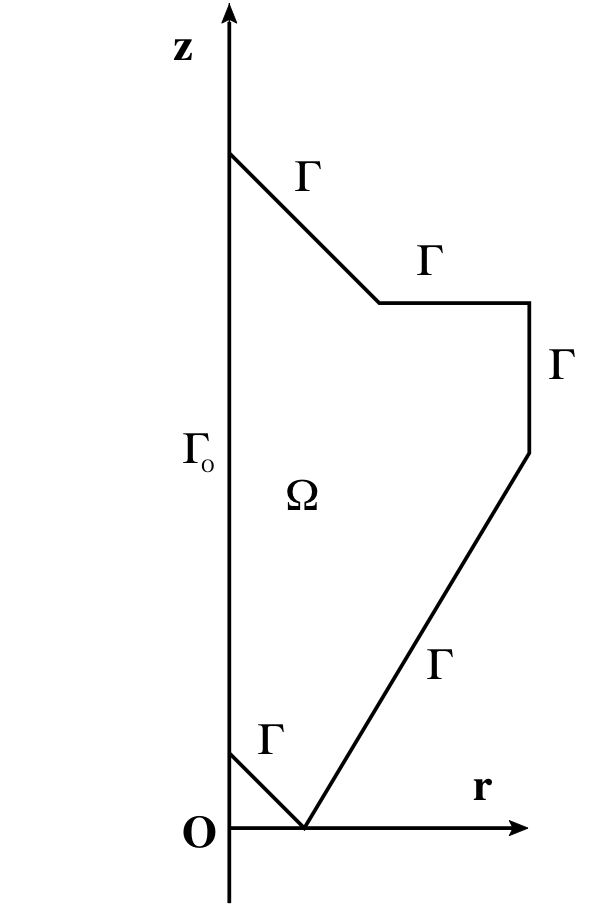}
\caption{An axisymmetric 3D domain $\tilde \Omega$ (left); the corresponding  2D polygonal domain $\Omega$ (right).}\label{fig.domain}
\end{figure}

Consider the 3D axisymmetric Stokes problem,
\begin{eqnarray}\label{eqn.2}
\left\{\begin{array}{ll}
-\Delta \tilde\bu+\nabla \tilde p =\tilde\bbf\quad {\rm{in}} \quad \tilde\Omega\\
{\rm{div}}\tilde\bu=0\quad {\rm{in}} \quad \tilde\Omega\\
\tilde\bu = 0 \quad {\rm{on}} \quad \pa\tilde\Omega,\\
\end{array}\right.
\end{eqnarray}
where $\tilde\bu$ and $\tilde\bbf$ (resp. $\tilde p$) are axisymmetric vector fields (resp. function) satisfying \eqref{eqn.1}. Assuming  the set $\pa\Omega\cap\{r=0\}$ has positive measure, we denote $\Gamma_0:=\pa\Omega\cap\{r=0\}$ and $\Gamma:=\pa\Omega\backslash\{r=0\}$ (Figure \ref{fig.domain}).  Then,  equation \eqref{eqn.2} can be reduced to a system of two decoupled equations \cite{book}:
the axisymmetric Stokes equations
\begin{eqnarray}\label{eqn.5}
\left\{\begin{array}{ll}
-(\pa_r^2+r^{-1}\pa_r+\pa_z^2-r^{-2})u_r+\partial_rp=f_r\quad {\rm{in}} \quad \Omega\\
-(\pa_r^2+r^{-1}\pa_r+\pa_z^2)u_z+\partial_zp=f_z\quad {\rm{in}} \quad \Omega\\
(\partial_r+r^{-1})u_r+\partial_zu_z=0\quad {\rm{in}} \quad \Omega\\
(u_r, u_z) = (0, 0) \quad {\rm{on}} \quad \Gamma,\\
\end{array}\right.
\end{eqnarray}
and the azimuthal Stokes equation
\begin{eqnarray}\label{eqn.6}
\left\{\begin{array}{ll}
\mathcal  -(\pa_r^2+r^{-1}\pa_r+\pa_z^2-r^{-2})u_\theta=f_\theta\quad {\rm{in}} \quad \Omega\\
u_\theta = 0 \quad {\rm{on}} \quad \Gamma.\\
\end{array}\right.
\end{eqnarray}

In this paper, we shall focus on the analysis and finite element approximation for the axisymmetric Stokes problem \eqref{eqn.5}. Numerical schemes for the azimuthal Stokes equation \eqref{eqn.6} shall be  studied in a forthcoming paper.
 Recall the polygonal domain $\Omega$ is in the $rz$-plane. We first adopt a class of weighted Sobolev spaces from \cite{book}.

\begin{definition}\label{def.regular} (Type I Weighted Spaces).
For an integer $m\geq 0$, define
\begin{eqnarray*} 
L^2_1(\Omega):=\{v,\ \int_\Omega v^2rdrdz<\infty\},\quad H^m_1(\Omega):=\{v,\ \pa_c^\alpha v\in L^2_1(\Omega),\  |\alpha|\leq  m\},
\end{eqnarray*}
where the muti-index $\alpha=(\alpha_1, \alpha_2)$ is a pair of nonnegative integers, $|\alpha|=\alpha_1+\alpha_2$, and $\pa^\alpha_c=\pa_r^{\alpha_1}\pa_z^{\alpha_2}$.
 The norms and the semi-norms for any $v\in H^m_1(\Omega)$ are  
\begin{eqnarray*}
\|v\|^2_{H^m_1(\Omega)}:=\sum_{|\alpha|\leq m}\int_{\Omega}(\partial_c^\alpha v)^2rdrdz, \quad
|v|^2_{H^m_1(\Omega)}:=\sum_{|\alpha|=m}\int_{\Omega}(\partial_c^\alpha v)^2rdrdz.
\end{eqnarray*}
Furthermore, we define two spaces $H^m_+(\Omega)$ and  $H^{m}_-(\Omega)$.

For $H^m_+(\Omega)$, if $m$ is not even, 
\begin{eqnarray}
&H^{m}_+(\Omega):=\{v\in H^m_1(\Omega),\ \pa_r^{2i-1}v|_{\{r=0\}}=0,\ 1\leq i<\frac{m}{2} \}, \label{eqn.s2}\\
&\|v\|_{H^{m}_+(\Omega)}=\|v\|_{H^m_1(\Omega)};\nonumber
\end{eqnarray}
if $m$ is even, besides the condition in \eqref{eqn.s2}, we require $ \int_{\Omega}(\partial_r^{m-1}v)^2r^{-1}drdz<\infty$
for any $v\in H^m_+(\Omega)$, and the corresponding norm is
$$\|v\|_{H^{m}_+(\Omega)}=\big(\|v\|^2_{H^m_1(\Omega)}+\int_{\Omega}(\partial_r^{m-1}v)^2r^{-1}drdz\big)^{1/2}.$$

For $H^m_-(\Omega)$, if $m$ is not odd,
\begin{eqnarray}
&H^{m}_-(\Omega):=\{v\in H^m_1(\Omega),\ \pa_r^{2i}v|_{\{r=0\}}=0,\ 0\leq i<\frac{m-1}{2} \} \label{eqn.s1}\\
&\|v\|_{H^{m}_-(\Omega)}=\|v\|_{H^m_1(\Omega)};\nonumber
\end{eqnarray}
if $m$ is odd, besides the condition in \eqref{eqn.s1}, we require 
$
 \int_{\Omega}(\partial_r^{m-1}v)^2r^{-1}drdz<\infty,
$ for any $v\in H^m_-(\Omega)$, and the corresponding norm is
\begin{eqnarray*}
\|v\|_{H^{m}_-(\Omega)}=\big(\|v\|^2_{H^m_1(\Omega)}+\int_{\Omega}(\partial_r^{m-1}v)^2r^{-1}drdz\big)^{1/2}.
\end{eqnarray*}

Thus, we   denote different subspaces:
\begin{eqnarray*}
&H^1_{1, 0}(\Omega):=H^1_1(\Omega)\cap\{v|_{\Gamma}=0\},  \quad H^1_{-, 0}(\Omega):=H^1_{-}(\Omega)\cap\{v|_{\partial \Omega}=0\}, \\ 
& H^1_{+, 0}(\Omega):=H^1_{+} (\Omega)\cap\{v|_{\Gamma}=0\}, \quad L^2_{1, 0}(\Omega):=L^2_1(\Omega)\cap\{v, \int_\Omega vrdrdz=0\}.
\end{eqnarray*}

\end{definition}

We now introduce another type of weighted spaces for our analysis on singular solutions of equation \eqref{eqn.5}.

\begin{definition}\label{def.regular1} (Type II Weighted Spaces). Let $Q_i$ be the $i$th  vertex of $\Omega$ and  define  the vertex set $\mathcal Q:=\{Q_i\}_{i=1}^{I}$.  Denote by $L$  the smallest distance from a vertex   to any disjoint edge of $\pa\Omega$. 
Let  $B(x, r_0)$ be the ball centered at $x$ with radius $r_0$. Let $\vartheta\in \mathcal C^{\infty}(\bar\Omega\backslash\mathcal Q)$ be a function, such  that  $\vartheta=|x-Q_i|$ in $\mathcal V_i:=\Omega\cap B(Q_i, L/2)$  and $\vartheta\geq  L/2$ in $\Omega\backslash \cup_{i=1}^{I} \mathcal V_i$. Note that $\maV_i$ and $\maV_j$ are disjoint if $i\neq j$. Thus, we define for $\mu\in\mathbb R$ and for any open set $G\subset\Omega$,
\begin{eqnarray*}
	\maK^{m}_{\mu, 1}(G) := \{v, \ \vartheta^{-\mu+|\alpha|}
	\partial_c^\alpha v\in L_1^2(G), \ |\alpha|
	\leq m\}.
\end{eqnarray*}
with the semi-norm and norm
\begin{eqnarray*}
 |v|^2_{\maK^{m}_{\mu, 1}(G)}:=\sum_{|\alpha|= m}\|\vartheta^{m-\mu}\partial_c^\alpha v\|^2_{L^2_1(G)}, \quad \|v\|^2_{\maK^{m}_{\mu, 1}(G)}:=\sum_{l=0}^m|v|^2_{\maK^{l}_{\mu, 1}(G)}.
\end{eqnarray*}
 Similarly, we define the subspaces  of $\maK^m_{\mu, 1}(\Omega)$: 
\begin{eqnarray}
&&\maK^{m}_{\mu, +}(\Omega):=\{v\in \maK^m_{\mu, 1}(\Omega),\ \int_{\Omega}\frac{\big(\pa_r^{2i-1}(\vartheta^{-\mu+l}v)\big)^2}{r}drdz<\infty,\ 1\leq i\leq \frac{l}{2}, \ 0\leq l\leq m \}, \label{eqn.ss2}\\
&&\maK^{m}_{\mu, -}(\Omega):=\{v\in \maK^m_{\mu, 1}(\Omega),\ \int_{\Omega}\frac{\big(\pa_r^{2i}(\vartheta^{-\mu+l}v))^2}{r}drdz<\infty,\ 0\leq i\leq \frac{l-1}{2},\ 0\leq l\leq m\}. \label{eqn.ss1}
\end{eqnarray}
The corresponding norms are
\begin{eqnarray*}
\|v\|_{\maK^{m}_{\mu, +}(\Omega)}=\Big(\|v\|^2_{\maK^{m}_{\mu, 1}(\Omega)}+\sum_{l\leq m}\sum_{1\leq i\leq \frac{l}{2}}\int_{\Omega}\big(\partial_r^{2i-1}(\vartheta^{-\mu+l}v)\big)^2r^{-1}drdz\Big)^{1/2}, \label{eqn.weight1}\\
\|v\|_{\maK^{m}_{\mu, -}(\Omega)}=\Big(\|v\|^2_{\maK^{m}_{\mu, 1}(\Omega)}+\sum_{l\leq m}\sum_{0\leq i\leq \frac{l-1}{2}}\int_{\Omega}\big(\partial_r^{2i}(\vartheta^{-\mu+l}v)\big)^2r^{-1}drdz\Big)^{1/2}. \label{eqn.weight2}
\end{eqnarray*}

\end{definition}

\begin{remark}
The solution of the azimuthal Stokes equation \eqref{eqn.6} is well-defined in $H^1_{-, 0}(\Omega)$ for $f_\theta\in H^1_{-, 0}(\Omega)'$ \cite{book}. Since it is completely decoupled from equation \eqref{eqn.5}, in the analysis below, we always set $f_\theta=0$ (and therefore $u_\theta=0$) in equation \eqref{eqn.6}. Namely, the angular components of the solution $\tilde\bu$ and the given data $\tilde\bbf$ vanish in the 3D stokes equation \eqref{eqn.2}. This will not affect our results on the axisymmetric Stokes problem, but simplify the exposition.
\end{remark}

Let $\tilde\bv$ (resp. $\tilde v$) be an axisymmetric  vector field (resp. function). Let  $\tilde{\mathbf H}^m(\tilde\Omega)\subset [H^{m}(\tilde\Omega)]^3$ (resp. $\tilde{H}^m(\tilde\Omega)\subset H^{m}(\tilde\Omega)$) be the subspace of axisymmetric  vector fields (resp. functions). We recall the following results from \cite{book}.

\begin{proposition}\label{prop.1}
The trace operator $\tilde v(x, y, z)\rightarrow v(r, z)$ defines the isomorphism
\begin{eqnarray*}
\tilde H^m(\tilde\Omega)\rightarrow H^m_+(\Omega);
\end{eqnarray*}
and the trace operator $\tilde{\bv}\rightarrow (v_r, v_\theta, v_z)$ defines the   isomorphism
\begin{eqnarray}\label{eqn.prop1}
\tilde{\mathbf{H}}^m(\tilde\Omega)\rightarrow H^m_-(\Omega)\times H^m_-(\Omega)\times H^m_+(\Omega),
\end{eqnarray}
where $v_r$, $v_\theta$, and $v_z$ are all axisymmetric functions.
\end{proposition}

\begin{remark}
Based on the well-posedness of the 3D Stokes problem and Proposition \ref{prop.1}, the right space for the solution $(u_r, u_z, p)$ of equation \eqref{eqn.5} is $H^1_-(\Omega)\times H^1_+(\Omega) \times L^2_1(\Omega)$. Note that we can obtain the boundary conditions on $\Gamma$  for equation  \eqref{eqn.5} by inheriting the boundary condition from the original 3D problem \eqref{eqn.2}. Based on Proposition 3.18 in \cite{ACL02}, $r^{-1}v\in L^2_1(\Omega)$ implies $v=0$ on the $z$-axis. This leads to  the zero boundary conditions on $\Gamma_0$  for $u_r$  by the definition of $H^m_-(\Omega)$. For a strong solution $u_z\in H^2_+(\Omega)$, the condition $r^{-1}\pa_ru_z\in L^2_1(\Omega)$ gives rise to the Neumann boundary condition $\pa_ru_z=0$ on $\Gamma_0$. These boundary conditions are due to the axisymmetry of the corresponding 3D vector field. Moreover, the  constraints on the integrals  in \eqref{eqn.ss2} and \eqref{eqn.ss1} imply $\pa_r^{2i-1}(\vartheta^{-\mu+l}v)|_{\Gamma_0}=0$ with $1\leq i\leq l/2$ and $0\leq l\leq m$  for $v\in \maK^m_{\mu, +}(\Omega)$ and $\pa_r^{2i}(\vartheta^{-\mu+l}v)|_{\Gamma_0}=0$ with $0\leq i\leq (l-1)/2$ and $0\leq l\leq m$ for $v\in \maK^m_{\mu, -}(\Omega)$.
\end{remark}

Thus, the variational formulation for the axisymmetric Stokes equation \eqref{eqn.5} is: Find $(\bu, p)\in H^1_{-,0}(\Omega)\times H^1_{+,0}(\Omega)\times L^2_{1,0}(\Omega)$, such that for any $(\bv, q)\in H^1_{-,0}(\Omega)\times H^1_{+,0}(\Omega)\times L^2_1(\Omega)$,
\begin{eqnarray}\label{eqn.4}
\left\{\begin{array}{ll}
a(\bu, \bv)+ b(\bv,   p) =\int_{\Omega}\bbf\cdot\bv\quad {\rm{in}} \quad\Omega\\
 b(\bu,  q)=0\quad {\rm{in}} \quad \Omega,
\end{array}\right.
\end{eqnarray}
where
\begin{eqnarray*}
a(\bu, \bv)= \int_{\Omega} (\nabla_c\bu: \nabla_c \bv+r^{-2}u_rv_r)rdrdz, \quad b(\bu, q)= -\int_{\Omega} (q\text{div}_c\bu+r^{-1}qu_r) rdrdz,
\end{eqnarray*}
$\bu=(u_r, u_z)^t$ and $\bbf=(f_r,f_z)^t$ as in \eqref{eqn.5}, $\text{div}_c\bu=\partial_ru_r+\partial_zu_z$, and $\nabla_c\bu$ is the matrix $(\partial_r,  \partial_z)^t\bu^t$. 

\begin{proposition}\label{lem.normal}
The weak formulation  \eqref{eqn.4} defines a unique solution $(\bu, p)\in H^1_{-,0}(\Omega)\times H^1_{+,0}(\Omega)\times L^2_{1, 0}(\Omega)$ for  $\bbf\in  H^1_{-,0}(\Omega)'\times H^1_{+,0}(\Omega)'$, and 
\begin{eqnarray}
\|u_r\|_{H^1_{-}(\Omega)}+\|u_z\|_{H^1_{+}(\Omega)}+\|p\|_{L^2_1(\Omega)}\leq C\|\bbf\|_{H^1_{-,0}(\Omega)'\times H^1_{+,0}(\Omega)'}.\label{ee2}
\end{eqnarray}
\end{proposition}
\begin{proof}
The well-posedness of equation \eqref{eqn.4} is given in \cite{BBD06, book, LL11}. The estimates in  \eqref{ee2} follows directly from the well-posedness.
\end{proof}

\begin{remark}
Type I weighted spaces are suitable to formulate the well-posedness result \eqref{ee2}. The regularity of the solution, however,  is determined by  the geometry of the domain and the singular coefficients in the differential operator, which greatly impacts the effectiveness of the numerical approximation. The new space in Definition \ref{def.regular1} resembles those in \cite{Kondratiev67,BNZ205, LMN10, MR1896262,MR2209521,MR0467032, Nicaise97} for singular solutions of standard elliptic problems. The additional constraints and vanishing weights on the $z$-axis is due to the axisymmetry in the data. We will show that higher regularity estimates can be formulated in these spaces, regardless of the singularity in the solution. 
\end{remark}

\subsection{Some lemmas}

We distinguish the vertices on the $z$-axis and away from the $z$-axis as follows. Each vertex $Q$ on the $z$-axis will be denoted by  $Q^z$; each $Q$ away from the $z$-axis will be denoted by  $Q^r$. Recall the neighborhood $\maV:=B(Q, L/2)\cap\Omega$ of the vertex $Q$. For a vertex $Q^z$, we denote by $\tilde{\mathcal V}:=\mathcal V\times [0, 2\pi)\subset\tilde\Omega$ the rotation of its neighborhood $\maV$ about the $z$-axis. On $\maV$ or $\tilde\maV$, we consider the new coordinate system that is a simple translation of the old $rz$- (or $xyz$-) coordinate system, now with the vertex  at the origin. Meanwhile,   we set a local
polar coordinate system $(\rho, \theta)$ on $\mathcal V$, where $Q$ is  the
origin, such that 
\begin{eqnarray}\label{polar}
(r, z)=(\rho\sin\theta, \rho\cos\theta).
\end{eqnarray} 
Namely, $\rho$ and $\theta$ are also the radius and the elevation angle, respectively,  in the  spherical coordinates on $\tilde{\mathcal V}$.  Recall the following relation between the new Cartesian coordinates  and the spherical coordinates $(\rho, \phi, \theta)$ on $\tilde\maV$,
\begin{eqnarray}\label{eqn.cartesian}
x=\rho\cos\phi\sin\theta, \quad y=\rho\sin\phi\sin\theta, \quad z=\rho\cos\theta.
\end{eqnarray}

 Throughout the paper,  by $H'$, we mean the dual space of $H$. As in Definition \ref{def.regular}, we  also use the multi-index $\alpha=(\alpha_1, \alpha_2, \alpha_3)$ for a 3D domain, such that $|\alpha|=\alpha_1+\alpha_2+\alpha_3$ and $\pa^\alpha=\pa_x^{\alpha_1}\pa_y^{\alpha_2}\pa_z^{\alpha_3}$. For two multi-indices $\alpha$ and $\beta$, we define $\alpha-\beta:=(\alpha_1-\beta_1, \alpha_2-\beta_2, \alpha_3-\beta_3)$. By $\beta<\alpha$ (resp. $\beta\leq \alpha$), we mean $\beta_i<\alpha_i$ (resp. $\beta_i\leq \alpha_i$), $i=1,2,3$. The generic constant $C>0$ in our analysis below may be different at different occurrences. It will depend on the computational domain, but not on the functions involved in the estimates or the mesh level in the finite element algorithms.



The following two lemmas contain useful weighted estimates in usual Sobolev spaces in the 3D neighborhood $\tilde\maV$ of a vertex $Q^z$.

\begin{lemma}\label{lem2}
Let  $(\rho, \phi, \theta)$ be the spherical coordinates on $\tilde{\mathcal V}\subset\tilde\Omega$, the neighborhood of a  vertex $Q^z$, with  $Q^z$ as the origin. Suppose $\sum_{|\alpha|\leq m}\|\rho^{-a+|\alpha|}\pa^\alpha v\|_{L^2(\tilde{\mathcal V})}^2<\infty$, for $m\geq 0$ and $a\in  \mathbb R$. Then, for any $0\leq l \leq m$,
\begin{eqnarray}\label{eqn.lem21}
\|\rho^{-a+l} v\|^2_{H^l(\tilde{\mathcal V})}\leq C\sum_{|\alpha|\leq m}\|\rho^{-a+|\alpha|}\pa^\alpha v\|_{L^2(\tilde{\mathcal V})}^2.
\end{eqnarray}
\end{lemma}

\begin{proof}
Note that for any $\nu\in \mathbb R$, 
\begin{eqnarray}\label{eqn.derivative}
\pa_{x_i}(\rho^{\nu}  v)=\nu x_i\rho^{\nu-2}  v+\rho^\nu\pa_{x_i} v,\quad {\rm{where}}\ x_i=x, y,\ {\rm{or}}\ z.
\end{eqnarray} 
 For $0\leq l' \leq l$, let   $\beta$ and $\beta'$ be two nonnegative integer multi-indices, such that $|\beta'|=l'$. Using the triangle inequality and \eqref{eqn.derivative}, we have
\begin{eqnarray*}
   \| \pa^{\beta'} \rho^{-a+l}  v\|_{L^2(\tilde{\maV})} &\leq& C\sum_{\alpha\leq \beta'}  \sum_{ \beta \leq \beta'-\alpha}\|x^{\beta_1}y^{\beta_2}z^{\beta_3}\rho^{-a+l-l'+|\alpha|-|\beta|}\pa^\alpha v\|_{L^2(\tilde{\maV})}\\
&\leq&  C\sum_{|\alpha| \leq l'}  \|\rho^{-a+l-l'+|\alpha|}\pa^\alpha v\|_{L^2(\tilde{\maV})} \leq C\sum_{|\alpha|\leq m}\|\rho^{-a+|\alpha|}\pa^\alpha v\|_{L^2(\tilde{\mathcal V})},
\end{eqnarray*}
where we used the relations between the spherical coordinates and the Cartesian coordinates in \eqref{eqn.cartesian}.

Then, we  have
\begin{eqnarray*}
 \| \pa^{\beta'} \rho^{-a+l} v\|_{L^2(\tilde{\maV})}^2
 \leq  C\sum_{|\alpha|\leq m}  \|\rho^{-a+|\alpha|}\pa^\alpha v\|^2_{L^2(\tilde{\maV})}.
\end{eqnarray*}
Summing up over all the possible $\beta'$'s, we have proved the estimate \eqref{eqn.lem21}. 
\end{proof}

\begin{lemma}\label{lem211}
Let  $a\in \mathbb R$ and the integer $0\leq l\leq m$. Let $\tilde{\mathcal V}$ be the neighborhood of a vertex $Q^z$ and let $(\rho,\phi, \theta)$ be its local spherical coordinates as defined in Lemma \ref{lem2}. 
Then,
\begin{eqnarray}\label{eqn.lem211}
\sum_{|\alpha|\leq m}\|\rho^{-a+|\alpha|}\pa^\alpha  v\|^2_{L^2(\tilde\maV)}\leq C\sum_{l\leq m}|\rho^{-a+l} v|_{H^l(\tilde\maV)}^2.
\end{eqnarray}
\end{lemma}

\begin{proof}

We prove it by induction. For $m=0$, 
$$
\|\rho^{-a}v\|^2_{L^2(\tilde\maV)}=\int_{\tilde\maV}\rho^{-2a} {v}^2dxdydz=|\rho^{-a}  v|^2_{H^0(\tilde\maV)}.
$$

Assume \eqref{eqn.lem211} holds for $m\geq 0$. We now prove for $m+1$. Let   $\beta$ and $\beta'$ be two nonnegative integer multi-indices. Then, for any $\alpha$ such that $|\alpha|=m+1$, by  \eqref{eqn.derivative} and the triangle inequality, we first have,
\begin{eqnarray*}
\|\rho^{-a+m+1}\pa^{\alpha}v\|_{L^2(\tilde \maV)}&\leq& \|\pa^{\alpha}(\rho^{-a+m+1}v)\|_{L^2(\tilde\maV)}\\
&&+C\sum_{|\beta'|\leq m}^{\beta'\leq \alpha}\sum_{\beta\leq\alpha-\beta'}\|x^{\beta_1}y^{\beta_2}z^{\beta_3}\rho^{-a+|\beta'|-|\beta|}\pa^{\beta'} v\|_{L^2(\tilde\maV)}\\
&\leq&  \|\pa^{\alpha}(\rho^{-a+m+1}v)\|_{L^2(\tilde\maV)}+C\sum_{|\beta'|\leq m}\|\rho^{-a+|\beta'|}\pa^{\beta'} v\|_{L^2(\tilde\maV)},
\end{eqnarray*}
where we also used the relations in \eqref{eqn.cartesian} in the last step.
Therefore, 
$$
\|\rho^{-a+m+1}\pa^{\alpha}v\|^2_{L^2(\tilde \maV)}\leq C(\|\pa^{\alpha}(\rho^{-a+m+1} v)\|_{L^2(\tilde \maV)}^2+\sum_{|\beta'|\leq m}\|\rho^{-a+|\beta'|}\pa^{\beta'} v\|_{L^2(\tilde \maV)}^2).
$$

Due to the assumption, \eqref{eqn.lem211} holds for $m$. Then, summing over all the possible $\alpha$'s, we therefore have 
$$
\sum_{|\alpha|=m+1}\|\rho^{-a+m+1}\pa^{\alpha}v\|^2_{L^2(\tilde \maV)}\leq C(|\rho^{-a+m+1}v|_{H^{m+1}(\tilde\maV)}^2+\sum_{l\leq m}|\rho^{-a+l}v|_{H^l(\tilde{\mathcal V})}^2).
$$
This proves \eqref{eqn.lem211} for $m+1$.
\end{proof}

Recall the multi-index $\alpha=(\alpha_1, \alpha_2)$ and the notation $\pa_c^\alpha$ from Definition \ref{def.regular}. Then, the following two lemmas concern the connection between the two types of weighted spaces in the 2D neighborhood $\maV$ in the $rz$-plane of a vertex $Q^z$.

\begin{lemma}\label{lem4}
As defined in \eqref{polar}, let  $(\rho,  \theta)$ be the polar coordinates on ${\mathcal V}\subset \Omega$, the neighborhood of a  vertex $Q^z$. Suppose $v\in\maK^m_{a,1 }(\maV)$. Then, for $0\leq l \leq m$ and $a\in\mathbb R$,
\begin{eqnarray}\label{eqn.lem24}
\|\rho^{-a+l} v\|^2_{H^l_1({\mathcal V})}\leq C\|v\|_{\maK_{a, 1}^m({\mathcal V})}^2.
\end{eqnarray}
\end{lemma}

\begin{proof}
Note that for any $\nu\in \mathbb R$, 
\begin{eqnarray}\label{eqn.deriv}
\pa_{x_i}(\rho^{\nu}  v)=\nu x_i\rho^{\nu-2} v+\rho^\nu\pa_{x_i} v,\quad {\rm{where}}\ x_i=r\  {\rm{or}}\ z. 
\end{eqnarray}
For $0\leq l' \leq l$, let   $\alpha$,  $\beta$, and $\beta'$ be three multi-indices, such that $|\alpha|=l'$. Using the triangle inequality and \eqref{eqn.deriv}, we have
\begin{eqnarray*}
   \| \partial^\alpha_c (\rho^{-a+l}  v)\|_{L_1^2({\maV})} &\leq& C\sum_{ \beta\leq \alpha}  \sum_{ \beta'\leq \alpha-\beta}\|r^{\beta_1'}z^{\beta_2'}\rho^{-a+l-l'+|\beta|-|\beta'|}\pa_c^\beta v\|_{L_1^2({\maV})}\\
&\leq&  C\sum_{|\beta|\leq l'}  \|\rho^{-a+l-l'+|\beta|}\pa_c^\beta v\|_{L^2_1({\maV})}\\
&\leq& C\sum_{|\beta|\leq m}\|\rho^{-a+|\beta|}\pa_c^\beta v\|_{L^2_1({\mathcal V})},
\end{eqnarray*}
where we also used the relations in \eqref{polar}. Therefore,
\begin{eqnarray*}
 \| \pa_c^\alpha (\rho^{-a+l} v)\|_{L^2_1({\maV})}^2
 \leq  C\|v\|^2_{\maK^m_{a,1}({\maV})}.
\end{eqnarray*}
Summing up over all the possible $\alpha$'s and $l'$'s, we have proved the estimate \eqref{eqn.lem24}. 
\end{proof}

\begin{lemma}\label{lem3}
Let  $a\in \mathbb R$ and the integer $0\leq l\leq m$. Let $\mathcal V$ be the neighborhood of a vertex $Q^z$ and $(\rho, \theta)$ be the polar coordinates on $\maV$ as in Lemma \ref{lem4}. Then,
\begin{eqnarray}\label{eqn.lem3}
\|v\|^2_{\maK^{m}_{a, 1}(\maV)}\leq C\sum_{l\leq m}|\rho^{-a+l}v|_{H^l_1(\mathcal V)}^2.
\end{eqnarray}
\end{lemma}

\begin{proof}

We prove it by induction. For $m=0$, 
$$
\|v\|^2_{\maK^0_{a, 1}(\maV)}=\int_{\maV}\rho^{-2a} v^2rdrdz=\|\rho^{-a} v\|^2_{H^0_1(\maV)}.
$$

Assume \eqref{eqn.lem3} holds for $m\geq 0$. 
We now prove for $m+1$. Let   $\alpha$,  $\beta$, and $\beta'$ be three multi-indices, such that $|\alpha|=m+1$. Then, using \eqref{eqn.deriv} and the triangle inequality, we first have 
\begin{eqnarray*}
\|\rho^{-a+m+1}\pa_c^\alpha v\|_{L^2_1(\maV)}&\leq& \|\pa_c^\alpha(\rho^{-a+m+1}v)\|_{L^2_1(\maV)}\\
&&+C\sum_{\beta'\leq \alpha}^{|\beta'|\leq m}\sum_{\beta\leq \alpha-\beta'}\|r^{\beta_1}z^{\beta_2}\rho^{-a+|\beta'|-|\beta|}\pa_c^{\beta'}v\|_{L^2_1(\maV)}\\
&\leq&  \|\pa_c^{\alpha}(\rho^{\-a+m+1}v)\|_{L^2_1(\maV)}+C\sum_{|\beta'|\leq m}\|\rho^{-a+|\beta'|}\pa_c^{\beta'}v\|_{L^2_1(\maV)},
\end{eqnarray*}
where we also used the relations in \eqref{polar} in the last step.
Therefore, 
$$
\|\rho^{-a+m+1}\pa_c^{\alpha}v\|^2_{L^2_1(\maV)}\leq C(\|\pa_c^\alpha(\rho^{-a+m+1}v)\|_{L^2_1(\maV)}^2+\sum_{|\beta'|\leq m}\|\rho^{-a+|\beta'|}\pa_c^{\beta'}v\|_{L^2_1(\maV)}^2).
$$

Due to the assumption, \eqref{eqn.lem3} holds for $m$. Summing over all the possible $\alpha$'s, we therefore have 
$$
|v|^2_{\maK^{m+1}_{a, 1}(\maV)}\leq C(|\rho^{-a+m+1}v|_{H^{m+1}_1(\maV)}^2+\sum_{l\leq m}|\rho^{-a+l}v|_{H^l_1(\mathcal V)}^2).
$$
This, together with the assumption, completes the proof.
\end{proof}

\section{Regularity estimates} \label{sec3}

We here summarize our regularity estimates for possible singular solutions of the axisymmetric Stokes equation \eqref{eqn.5} in weighted Sobolev spaces. We shall also show the calculation of the index $\eta$, such that the solution does not lose regularity in these spaces.

\subsection{Local estimates}

The first estimate concerns the local behavior of the solution of the axisymmetric Stokes equation \eqref{eqn.5} in the neighborhood of a vertex away from the $z$-axis.

\begin{lemma}\label{3.2}
In the neighhood $\maV$ of a vertex $Q^r$ away from the $z$-axis, the solution $\bu=(u_r, u_z)$ satisfies
\begin{eqnarray*}
\|\vartheta^{-1}u_r\|_{L^2_1(\maV)}\leq C\|u_r\|_{H^1_{-}(\maV)},\qquad \|\vartheta^{-1}u_z\|_{L^2_1(\maV)}\leq C\|u_z\|_{H^1_{+}(\maV)},
\end{eqnarray*}
where $\vartheta$ is the function in Definition \ref{def.regular1}. 
\end{lemma}
\begin{proof}

 On $\maV$, both $H^1_{+}$ and $H^1_-$ (resp. $L^2_1$) are equivalent to the usual Sobolev space $H^1$ (resp. $L^2$), since $r$ is bounded away from $0$. Therefore, it suffices to show for any $v\in H^1(\maV)\cap \{v|_{\Gamma}=0\}$, 
\begin{eqnarray}\label{above}
\|\rho^{-1} v\|_{L^2(\maV)}\leq C\|v\|_{H^1(\maV)},
\end{eqnarray}
where $\rho$ is the distance to $Q^r$. However, the estimate in \eqref{above} is well known based on a local Poincar\'e inequality. See \cite{KMR01, KMR197,LN09, BNZ107}. 
\end{proof}

We now have the following estimates on the local property of the solution of the 3D Stokes problem \eqref{eqn.2} near a vertex on the $z$-axis.

\begin{lemma}\label{3.3}
Let $\tilde\bu=(u_1, u_2, u_3)\in [H^1_0(\tilde\Omega)]^3$ be the solution of the 3D Stokes problem and let $\tilde{\maV}=\maV\times[0, 2\pi)$ be the 3D neighborhood  of a vertex $Q^z$ on the $z$-axis. Then, each $u_j$, $1 \leq j\leq 3$, satisfies 
\begin{eqnarray*}
\|\vartheta^{-1}u_j\|_{L^2(\tilde{\maV})}\leq C\|u_j\|_{H^1(\tilde{\maV})},
\end{eqnarray*}
where $\vartheta$ is the distance function to the vertex $Q^z$.
\end{lemma}
\begin{proof}
$\tilde{\maV}$ can be characterized in the spherical coordinates $(\rho, \phi, \theta)$ centered at $Q^z$ by
\begin{eqnarray*}
\tilde{\maV}=\{(\rho, \omega),\ 0<\rho<L/2,\ \omega\in \omega_{Q^z}\},
\end{eqnarray*}
where $\omega_{Q^z}\subset S^2$ is the polygonal domain on the unit sphere $S^2$. Then, for any $v\in H^1(\tilde\maV)\cap\{v|_{\pa\tilde\Omega}=0\}$,
$$
|\nabla v|^2=v_x^2+v_y^2+v_z^2=v_\rho^2+\frac{v_{\theta}^2}{\rho^2}+\frac{v_{\phi}^2}{\rho^2\sin^2\theta},
$$
and 
$$
\int_{\omega_{Q^z}}v^2dS\leq C\int_{\omega_{Q^z}}(v_{\theta}^2+\frac{v_{\phi}^2}{\sin^2\theta})\sin\theta d\phi d\theta,
$$
which is just the Poincar\'e inequality on $\omega_{Q^z}$ and $dS=\sin{\theta}d\phi d\theta$ is the volume element on $\omega_{Q^z}$. Thus, we obtain
\begin{eqnarray}
\int_{\tilde{\maV}}\frac{v^2}{\rho^2}dxdydz&=&\int_0^{L/2}\int_{\omega_{Q^z}}v^2dS d\rho \leq C\int_0^{L/2}\int_{\omega_{Q^z}}\big(v_\rho^2+\frac{v_{\theta}^2}{\rho^2}+\frac{v_{\phi}^2}{\rho^2\sin^2\theta}\big)\rho^2dSd\rho \nonumber\\
&=&C\int_{\tilde{\maV}}|\nabla v|^2dxdydz.\label{above1}.
\end{eqnarray}
The estimate \eqref{above1} is valid for all functions $u_1$, $u_2$ and $u_3$, which completes the proof. 
\end{proof}

Recall the neighborhoods $\maV$ (2D) and $\tilde\maV$ (3D) of a vertex. Define the small neighborhoods $\maV/k:=\Omega\cap B(Q, L/(2k))\subset\maV$ and $\tilde\maV/k=\maV/k\times[0, 2\pi)\subset \tilde{\maV}$, where the integer $k\geq 1$. We first have the local regularity estimate for the solution of the axisymmetric Stokes equation near a vertex away from the $z$-axis.

\begin{lemma}\label{lem.2dstokes}
Near a vertex $Q^r$ away from the $z$-axis, there exists $\eta>0$, such that for any $0\leq a<\eta$, if $\bbf\in [\maK^{m}_{a-1, 1}(\maV)]^2$, the solution $(\bu, p)$ of equation \eqref{eqn.5} satisfies 
\begin{eqnarray*}
\|\bu\|_{[\maK^{m+2}_{a+1, 1}(\maV/2)]^2}+\|p\|_{\maK^{m+1}_{a, 1}(\maV/2)}\leq C(\|\bbf\|_{[\maK^{m}_{a-1, 1}(\maV)]^2}+\|\bbf\|_{H^1_{-,0}(\Omega)'\times H^1_{+,0}(\Omega)'}).
\end{eqnarray*}
\end{lemma}

\begin{proof}

We  apply a localization argument. Let $\zeta$ be a smooth cutoff function, such that $\zeta=1$ on $\maV/2$ and $\zeta=0$ outside $\maV$. Then, $\zeta\bu$ has the Dirichlet boundary condition on $\maV$. Then, we have 
\begin{eqnarray}\label{18}
\left\{\begin{array}{ll}
-(\pa_r^2+r^{-1}\pa_r+\pa_z^2-r^{-2})\zeta u_r+\partial_r\zeta p=F_1:=\zeta f_r+g_1 \quad {\rm{in}}\ \maV\\
-(\pa_r^2+r^{-1}\pa_r+\pa_z^2)\zeta u_z+\partial_z\zeta p=F_2:=\zeta f_z+g_2\quad {\rm{in}}\ \maV\\
(\partial_r+r^{-1})\zeta u_r+\partial_z\zeta u_z=g_3\quad {\rm{in}}\ \maV,
\end{array}\right.
 \end{eqnarray}
where 
\begin{eqnarray*}
&g_1=u_r(\pa_r^2+\pa_z^2)\zeta+2(\pa_r\zeta\pa_ru_r+\pa_z\zeta\pa_zu_r)+r^{-1}u_r\pa_r\zeta-p\pa_r\zeta,\\
&g_2=u_z(\pa_r^2+\pa_z^2)\zeta+2(\pa_r\zeta\pa_ru_z+\pa_z\zeta\pa_zu_z)+r^{-1}u_z\pa_r\zeta-p\pa_z\zeta,\\
&g_3=u_r\pa_r\zeta+u_z\pa_z\zeta.
\end{eqnarray*}
Based on Proposition \ref{lem.normal} and Lemma \ref{3.2}, the solution $(\zeta\bu, \zeta p)$ of eqution \eqref{18} satisifes
\begin{eqnarray}
&&\|\vartheta^{-1}\zeta u_r\|)_{L^2_1(\maV)}+\|\vartheta^{-1}\zeta u_z\|_{L^2_1(\maV)}+\|\zeta u_r\|_{H^1_{-}(\maV)}+\|\zeta u_z\|_{H^1_{+}(\maV)}+\|\zeta p\|_{L^2_1(\maV)}\label{eqn.25}\\
&&\leq C(\|\zeta u_r\|_{H^1_{-}(\maV)}+\|\zeta u_z\|_{H^1_{+}(\maV)}+\|\zeta p\|_{L^2_1(\maV)})\nonumber\\ 
&&\leq C(\|\zeta f_r\|_{H^1_{-,0}(\maV)'}+\|g_1\|_{H^1_{-,0}(\maV)'}+\|\zeta f_z\|_{H^1_{+,0}(\maV)'}+\|g_2\|_{H^1_{+,0}(\maV)'}+\|g_3\|_{L^2_1(\maV)'}). \nonumber
\end{eqnarray}
Recall $r$ is bounded away from 0 on $\maV$. Then, the regularity of the solution \eqref{18} is determined by the principle part of the operator, which is the 2D Stokes operator.
Also note that the supports of $g_1$, $g_2$, and $g_3$ are away from the vertex $Q^r$. Therefore, the weighted norms and the usual Sobolev norms are equivalent for these functions. Let $\bbg=(g_1, g_2, g_3)$. Then, using the interior regularity estimate in the usual Sobolev spaces and Proposition \ref{lem.normal}, we have
\begin{eqnarray}
\|\bbg\|_{[\maK^{m}_{a-1, 1}(\maV)]^2\times \maK^{m+1}_{a, 1}(\maV)}&\leq &C\|\bbg\|_{[H^{m}(\maV)]^2\times H^{m+1}(\maV)}\nonumber \\
&\leq& C(\|\bu\|_{[H^{m+1}(\maV\backslash\maV/2)]^2}+\|p\|_{H^m(\maV\backslash\maV/2)})\nonumber\\
&\leq& C(\|\bbf\|_{[H^{m-1}(\maV\backslash\maV/4)]^2}+\|\bu\|_{[H^{1}(\maV\backslash\maV/4)]^2}+\|p\|_{L^2(\maV\backslash\maV/4)})\nonumber\\
&\leq &C(\|\bbf\|_{[H^{m-1}(\maV\backslash\maV/4)]^2}+\| \bbf\|_{H^1_{-,0}(\Omega)'\times H^1_{+,0}(\Omega)'})\nonumber\\
&\leq& C(\|\bbf\|_{[\maK^{m}_{a-1, 1}(\maV)]^2}+\| \bbf\|_{H^1_{-,0}(\Omega)'\times H^1_{+,0}(\Omega)'}).  \label{rhs1}
\end{eqnarray}
Thus, by \eqref{rhs1}, the right hand side of equation \eqref{18} $(F_1, F_2, g_3)\in \maK^{m}_{a-1, 1}(\maV)\times\maK^{m}_{a-1, 1}(\maV)\times \maK^{m+1}_{a, 1}(\maV)$.

Let $\maH$ be either $H^1_{+, 0}(\maV)'$ or $H^1_{-, 0}(\maV)'$. Since $r$ is bounded away from 0,  $\maH=H^{-1}(\maV)$. Then, for any $v\in \maH$, by Lemma \ref{3.2} and the fact $a\geq 0$, we have
\begin{eqnarray}
\|v\|_\maH&=&\sup_{0\neq w\in H^1_0(\maV)}\frac{(v, w)}{\|w\|_{H^1(\maV)}}\leq \sup_{0\neq w\in H^1_0(\maV)}\frac{\|\vartheta^{1-a}v\|_{L^2(\maV)}\|\vartheta^{a-1}w\|_{L^2(\maV)}}{\|w\|_{H^1(\maV)}} \label{27}\\
&\leq &\|\vartheta^{1-a}v\|_{L^2(\maV)}=\|v\|_{\maK^{0}_{a-1, 1}(\maV)}. \nonumber
\end{eqnarray}

For $m=0$, setting $a=0$ in \eqref{eqn.25}, \eqref{rhs1}, and \eqref{27}, we then have 
\begin{eqnarray*}
&&\|\zeta u_r\|_{\maK^1_{1, 1}(\maV)}+\|\zeta u_z\|_{\maK^1_{1, 1}(\maV)}+\|\zeta p\|_{\maK^0_{0, 1}(\maV)}\\
&&\leq C(\|\zeta f_r+g_1\|_{\maK^0_{-1, 1}(\maV)}+\|\zeta f_z+g_2\|_{\maK^0_{-1, 1}(\maV)}+\|g_3\|_{\maK^1_{0, 1}(\maV)}).
\end{eqnarray*}
Let $\omega>0$ be  the least positive real part of the eigenvalues of the operator pencil for the 2D Stokes operator on $\maV$ \cite{Dauge89}. Define $\eta:=\omega$. Based on Corollary 1.2.7 in \cite{mazyabook}, if the solution of equation \eqref{18} is in $\maK^1_{1, 1}(\maV)\times\maK^1_{1, 1}(\maV)\times\maK^0_{0, 1}(\maV)$ and $(F_1, F_2, g_3)\in \maK^m_{a-1, 1}(\maV)\times\maK^m_{a-1, 1}(\maV)\times\maK^{m+1}_{a, 1}(\maV)$, as long as $0\leq a<\eta$, we can conclude
\begin{eqnarray*}
&&\|\zeta u_r\|_{\maK^{m+2}_{a+1, 1}(\maV)}+\|\zeta u_z\|_{\maK^{m+2}_{a+1, 1}(\maV)}+\|\zeta p\|_{\maK^{m+1}_{a, 1}(\maV)}\\
&&\leq C(\|\zeta f_r+g_1\|_{\maK^{m}_{a-1, 1}(\maV)}+\|\zeta f_z+g_2\|_{\maK^m_{a-1, 1}(\maV)}+\|g_3\|_{\maK^{m+1}_{a, 1}(\maV)})\\
&&\leq C(\|\bbf\|_{[\maK^{m}_{a-1, 1}(\maV)]^2}+\| \bbf\|_{H^1_{-,0}(\Omega)'\times H^1_{+,0}(\Omega)'}). 
\end{eqnarray*}
The lemma is thus proved due to the definition of the function $\zeta$.
\end{proof}

We now give a regularity estimate near a vertex on the $z$-axis.
\begin{lemma}\label{lem.3dstokes}
In the small neighborhood $\tilde{\maV}\subset\tilde\Omega$ of a vertex $Q^z$ on the $z$-axis, there is $\eta>0$, such that for any $0\leq a<\eta$,  the solution $(\tilde u_1, \tilde u_2, \tilde u_3, \tilde p)$ of the 3D Stokes equation \eqref{eqn.2} satisfies 
\begin{eqnarray*}
\big(\sum_{k=1}^3\sum_{|\alpha|\leq{m+2}}\|\vartheta^{-a-1+|\alpha|}\pa^{\alpha}\tilde u_k\|^2_{L^2({\tilde{\maV}/2})})^{1/2}+\big(\sum_{|\alpha|\leq{m+1}}\|\vartheta^{-a+|\alpha|}\pa^{\alpha}\tilde p\|^2_{L^2(\tilde{\maV}/2)}\big)^{1/2}\\
\leq C\Big(\big(\sum_{k=1}^3\sum_{|\alpha|\leq{m}}\|\vartheta^{-a+1+|\alpha|}\pa^{\alpha}\tilde f_k\|^2_{L^2(\tilde{\maV})}\big)^{1/2}+\|\tilde\bbf\|_{[H^{-1}(\tilde\Omega)]^3}\Big).
\end{eqnarray*}
\end{lemma}
\begin{proof}
We use a localization augment similarly to the one in Lemma \ref{lem.2dstokes}. Let $\zeta$ be a smooth cutoff function, such that $\zeta=1$  on $\tilde{\maV}/2$ and $\zeta = 0$ outside $\tilde{\maV}$. Let $S$ be the 3D Stokes operator in equation \eqref{eqn.2}. Then, we have
\begin{eqnarray}\label{newn}
S(\zeta\tilde u, \zeta\tilde p)=(\zeta \tilde \bbf + \tilde\bbh, \tilde g),
\end{eqnarray}
where 
\begin{eqnarray*}
&\tilde \bbh=\tilde \bu\Delta \zeta+2\pa_x\tilde \bu\pa_x\zeta+2\pa_y\tilde \bu\pa_y\zeta+2\pa_z\tilde \bu\pa_z\zeta-\tilde p\nabla\zeta,\\
&\tilde g=\tilde u_1\pa_x\zeta+\tilde u_2\pa_y\zeta+\tilde u_3\pa_z\zeta.
\end{eqnarray*}
Since $\tilde\bbh=(\tilde h_1, \tilde h_2, \tilde h_3)$ and $\tilde g$ vanish near $Q^z$, using the well-posedness of the Stokes problem \eqref{eqn.2}, the usual interior regularity estimate, and the expressions of $\tilde\bbh$, $\tilde g$ above, we first have
\begin{eqnarray}
&&\big(\sum_{k=1}^3\sum_{|\alpha|\leq{m}}\|\vartheta^{-a+1+|\alpha|}\pa^{\alpha}\tilde h_k\|_{L^2(\tilde{\maV})}^2\big)^{1/2}+\big(\sum_{|\alpha|\leq{m+1}}\|\vartheta^{-a+|\alpha|}\pa^{\alpha}\tilde g\|^2_{L^2(\tilde{\maV} )}\big)^{1/2}\nonumber\\
&&\leq C(\|\tilde \bu\|_{[H^{m+1}(\tilde\maV \backslash\tilde \maV /2)]^3}+\|\tilde p\|_{H^{m}(\tilde\maV \backslash\tilde \maV /2)})\leq C(\|\tilde\bbf\|_{[H^{m-1}(\tilde\maV \backslash\tilde \maV /4)]^3}+\|\tilde\bbf\|_{[H^{-1}(\tilde\Omega)]^3})\nonumber\\
&&\leq C\Big(\big(\sum_{k=1}^3\sum_{|\alpha|\leq{m}}\|\vartheta^{-a+1+|\alpha|}\pa^{\alpha}\tilde f_k\|^2_{L^2(\tilde{\maV} )}\big)^{1/2}+\|\tilde\bbf\|_{[H^{-1}(\tilde\Omega)]^3}\Big).\label{rhs2}
\end{eqnarray}
Therefore, the right hand side of equation \eqref{newn} is bounded by \eqref{rhs2}.

For $m=0$ and $a=0$, by Lemma \ref{3.3}, \eqref{rhs2}, and the well-posedness of the local Stokes problem \eqref{newn}, we have
\begin{eqnarray}
&&\big(\sum_{k=1}^3\sum_{|\alpha|\leq{1}}\|\vartheta^{-1+|\alpha|}\pa^{\alpha}(\zeta\tilde u_k)\|_{L^2(\tilde{\maV} )}^2\big)^{1/2}+\|\zeta\tilde p\|_{L^2(\maV )}\label{3dw}\\
&&\leq C\big(\|\zeta\tilde \bbf+\tilde\bbh\|_{[H^{-1}(\tilde\maV)]^3}+\|\tilde g\|_{L^2(\tilde\maV)}\big) \leq C\Big(\big(\sum_{k=1}^3\|\vartheta\tilde f_k\|^2_{L^2(\tilde{\maV})}\big)^{1/2}+\|\tilde\bbf\|_{[H^{-1}(\tilde\Omega)]^3}\Big).\nonumber
\end{eqnarray}

Let $\omega>0$ be the least positive real part of the  eigenvalues of the operator pencil for the 3D Stokes operator in \eqref{newn}. Define $\eta=\omega+1/2$. Based on Corollary 1.2.7 in \cite{mazyabook}, the estimate in \eqref{3dw} and \eqref{rhs2} imply
\begin{eqnarray*}
\big(\sum_{k=1}^3\sum_{|\alpha|\leq{m+2}}\|\vartheta^{-a-1+|\alpha|}\pa^{\alpha}(\zeta\tilde u_k)\|^2_{L^2({\tilde{\maV}})})^{1/2}+\big(\sum_{|\alpha|\leq{m+1}}\|\vartheta^{-a+|\alpha|}\pa^{\alpha}(\zeta\tilde p)\|_{L^2(\tilde{\maV})}\big)^{1/2}\\
\leq C\Big(\big(\sum_{k=1}^3\sum_{|\alpha|\leq{m}}\|\vartheta^{-a+1+|\alpha|}\pa^{\alpha}\tilde f_k\|^2_{L^2(\tilde{\maV})}\big)^{1/2}+\|\tilde\bbf\|_{[H^{-1}(\tilde\Omega)]^3}\Big).
\end{eqnarray*}
 as long as $0\leq a<\eta$.
 
 The lemma is thus proved due to the definition of $\zeta$.
\end{proof}

\subsection{Global estiamtes}
Combining the local estimates in the lemmas above, we derive the global regularity estimate for equation \eqref{eqn.4}.

\begin{theorem}\label{thm.stokesreg} Let $(\bu, p)\in H^1_{-,0}(\Omega)\times H^1_{+,0}(\Omega)\times L^2_{1, 0}(\Omega)$ be the solution of the axisymmetric Stokes equation \eqref{eqn.4}. There exists $\eta>0$, such that for any $0\leq a<\eta$, if $\bbf\in \maK^{m}_{a-1, -}(\Omega)\times \maK^{m}_{a-1, +}(\Omega)$, then 
\begin{eqnarray*}
\|\vartheta^{-a}r^{-1}u_r\|_{L^2_1(\Omega)}+\|\bu\|_{[\maK^{m+2}_{a+1, 1}(\Omega)]^2}+\|p\|_{\maK^{m+1}_{a, 1}(\Omega)} \leq C\|\bbf\|_{\maK^{m}_{a-1, -}(\Omega)\times \maK^{m}_{a-1, +}(\Omega)}.
\end{eqnarray*}
\end{theorem}

\begin{proof}
Let $\omega_i>0$ be  the least positive real part of the eigenvalues of the operator pencil  for the Stokes operator in the neighborhood of the vertex $Q_i$ as in Lemmas \ref{lem.2dstokes} and \ref{lem.3dstokes}. Let 
\begin{eqnarray}
\eta_i=\omega_i, \qquad {\rm{if}}\ Q_i\notin \{r=0\};\label{1111}\\
\eta_i=\omega_i+1/2, \qquad {\rm{if}}\ Q_i\in \{r=0\}.\label{2222}
\end{eqnarray}
Define
\begin{eqnarray}\label{eta}
\eta:=\min_i(\eta_i).
\end{eqnarray} 

Recall that the weighted  space $\maK^{m}_{a, 1}$ (resp. $\maK^m_{a, +}$ and $\maK^m_{a, -}$) is equivalent to the weighted space $H^m_{1}$ (resp. $H^m_{+}$ and $H^m_{-}$) in a subdomain  $\Omega_{sub}\subset\Omega$ that is away from the vertex set. Based on the isomorphism in \eqref{eqn.prop1} and the well-posedness  and the usual interior regularity estimate for the 3D Stokes problem, we have
\begin{eqnarray}
\|\vartheta^{-a}r^{-1}u_r\|_{L^2_1(\Omega_{sub})}+\|\bu\|_{[\maK^{m+2}_{a+1, 1}(\Omega_{sub})]^2}+\|p\|_{\maK^{m+1}_{a, 1}(\Omega_{sub})} \leq C\|\tilde\bbf\|_{[H^{m}(\tilde\Omega')]^3}+\|\tilde\bbf\|_{H^{-1}(\tilde\Omega)}\nonumber\\
\leq C(\|\bbf\|_{\maK^{m}_{a-1, -}(\Omega')\times \maK^{m}_{a-1, +}(\Omega')}+\|\bbf\|_{\maK^{0}_{a-1, -}(\Omega)\times \maK^{0}_{a-1, +}(\Omega)}), \label{eqn.subset}
\end{eqnarray}
where $\Omega_{sub}\subset\subset\Omega'\subset\subset \Omega$ and $\tilde\Omega'=\Omega'\times[0, 2\pi)$ is from the rotation of $\Omega'$ about the $z$-axis.

Let $\maV$ be the neighborhood of a vertex $Q^r$ away from the $z$-axis. By Lemma \ref{lem.2dstokes} and the fact that $r$ is bounded away from 0 on $\maV$,
\begin{eqnarray}\label{eqn.qa}
&&\|\vartheta^{-a}r^{-1}u_r\|_{L^2_1(\maV/2)}+\|\bu\|_{[\maK^{m+2}_{a+1, 1}(\maV/2)]^2}+\|p\|_{\maK^{m+1}_{a, 1}(\maV/2)}\\
&&\leq C(\|\bbf\|_{[\maK^{m}_{a-1, 1}(\maV)]^2}+\|\bbf\|_{H^1_{-,0}(\Omega)'\times H^1_{+,0}(\Omega)'})\nonumber\\
&&\leq C(\|\bbf\|_{\maK^{m}_{a-1, -}(\maV)\times\maK^{m}_{a-1, +}(\maV)}+\|\bbf\|_{\maK^{0}_{a-1, -}(\Omega)\times \maK^{0}_{a-1, +}(\Omega)}).
\end{eqnarray}

We now show the estimates in $\maV$, the small neighborhood of a vertex $Q^z$ on the $z$-axis. By Lemma \ref{lem4}, we first have for any $0\leq l\leq m$,
\begin{eqnarray}\label{eqn.step1}
\|\rho^{1-a+l}\bbf\|_{[H^l_1(\maV)]^2}\leq C\|\bbf\|_{[\maK^m_{a-1, 1}(\maV)]^2}.
\end{eqnarray}
Then, for $f_r\in \maK^m_{a-1, -}(\maV)$, \eqref{eqn.step1} and the condition in \eqref{eqn.ss1}
$$
 \int_{\Omega}\big(\pa_r^{2i}(\vartheta^{1-a+l}f_r)\big)^2r^{-1}drdz<\infty, \quad 0\leq i\leq (l-1)/2
$$
lead to $\rho^{1-a+l}f_r\in H^l_{-}(\maV)$. Similarly, using \eqref{eqn.step1} and the condition in \eqref{eqn.ss2}, we conclude that for $f_z\in \maK^m_{a-1, +}(\maV)$,  $\rho^{1-a+l}f_r\in H^l_{+}(\maV)$. Then, by Lemma \ref{lem211}, the isomorphism  in \eqref{eqn.prop1}, and the definitions of the weighted spaces in \eqref{eqn.s2}, \eqref{eqn.s1}, \eqref{eqn.ss2}, and \eqref{eqn.ss1},
\begin{eqnarray}\label{eqn.prev}
&&\sum_{|\alpha |\leq m}\|\rho^{1-a+|\alpha|}\pa^\alpha\tilde \bbf\|^2_{[L^2(\tilde\maV)]^3}\leq C\sum_{0\leq l\leq m}\|\rho^{1-a+l}\tilde\bbf\|^2_{[H^l(\tilde\maV)]^3}\nonumber\\
&&\leq C \sum_{0\leq l\leq m}\|\rho^{1-a+l}\bbf\|^2_{H^l_-(\maV)\times H^l_+(\maV)}\leq C\|\bbf\|^2_{\maK^m_{a-1, -}(\maV)\times \maK^m_{a-1, +}(\maV)}.
\end{eqnarray}
Then, by Lemma \ref{lem3},   the isomorphism in \eqref{eqn.prop1}, Lemma \ref{lem2}, Lemma \ref{lem.3dstokes}, and \eqref{eqn.prev}, we have
\begin{eqnarray}
&&\|\vartheta^{-a}r^{-1}u_r\|_{L^2_1(\maV/2)}+\|\bu\|_{[\maK^{m+2}_{a+1, 1}(\maV/2)]^2}+\|p\|_{\maK^{m+1}_{a, 1}(\maV/2)}\nonumber\\
&&\leq C\big(\|\vartheta^{-a}r^{-1}u_r\|_{L^2_1(\maV/2)}+(\sum_{l\leq m+2}\|\rho^{-a-1+l}\bu\|^2_{[H^l_1(\maV/2)]^2})^{1/2}+(\sum_{l\leq m+1}\|\rho^{-a+l}p\|^2_{H^l_1(\maV/2)})^{1/2}\big)\nonumber\\
&&\leq C\big((\sum_{l\leq m+2}\|\rho^{-a-1+l}\tilde\bu\|^2_{[H^l(\tilde\maV/2)]^3})^{1/2}+(\sum_{l\leq m+1}\|\rho^{-a+l}\tilde p\|^2_{H^l(\tilde\maV/2)})^{1/2}\big)\nonumber\\
&&\leq C\big((\sum_{k=1}^3\sum_{|\alpha|\leq m+2}\|\rho^{-a-1+|\alpha|}\pa^{\alpha}\tilde u_k\|^2_{[L^2(\tilde\maV/2)]^3})^{1/2}+(\sum_{|\alpha|\leq m+1}\|\rho^{-a+|\alpha|}\pa^{\alpha}\tilde p\|^2_{L^2(\tilde\maV/2)})^{1/2}\big)\nonumber\\
&&\leq C\Big(\big(\sum_{k=1}^3\sum_{|\alpha|\leq{m}}\|\vartheta^{-a+1+|\alpha|}\pa^{\alpha}\tilde f_k\|^2_{L^2(\tilde{\maV})}\big)^{1/2}+\|\vartheta^{-a+1}\tilde\bbf\|_{[L^2(\tilde\Omega)]^3}\Big)\nonumber\\
&&\leq C(\|\bbf\|_{\maK^m_{a-1, -}(\maV)\times \maK^m_{a-1, +}(\maV)}+\|\bbf\|_{\maK^0_{a-1, -}(\Omega)\times \maK^0_{a-1, +}(\Omega)}.\label{eqn.3rd}
\end{eqnarray}

The proof is completed by combining \eqref{eqn.subset}, \eqref{eqn.qa}, and \eqref{eqn.3rd}.
\end{proof}
\begin{remark}
Note that the regularity estimate in Theorem \ref{thm.stokesreg} is up to any order depending on the regularity of the given data. The calculation of the local indices in \eqref{1111}, \eqref{2222} and the global index in \eqref{eta} will also be useful to justify our optimal finite element approximation in Section \ref{sec4}.
\end{remark}

\section{The finite element approximation} \label{sec4}
We discuss the  finite element approximation of the axisymmetric equation \eqref{eqn.5} in polygonal domains. We are aware that a few mixed finite element formulations have proved to be stable (e.g., P1isoP2/P1 and Taylor-Hood elements) for our target problem \cite{BBD06, LL11}. Since the solution may present different singularities near vertices on or away from the $z$-axis, the approximation properties of these methods similarly depend on the regularity of the solution and the best approximations from the discrete subspaces. 
Our focus, rather than the stability issue of the mixed methods, will be on the construction of speical finite element spaces that provide numerical solutions with optimal convergence rates in the presence of singular solutions of equation \eqref{eqn.5}. Although our approach applies to other mixed methods, to simplify the presentation, we in particular concentrate on the Taylor-Hood mixed method. 

\subsection{The mixed forumulation}
Let $\maT_n=\{T_i\}$ be a triangulation of the domain $\Omega$ with triangles $T_i$. For a bounded domain $G\subset\mathbb R^2$, let $\maP^k(G)$ be the space of polynomials of degree $k$ on $G$.  We denote the space of continuous piecewise polynomials of degree $k$, associated to the triangulation $\maT_n$, by
\begin{eqnarray}\label{eqn.p}
P^k(\Omega)=\{p\in \maC^0(\Omega),\ p|_T\in \maP^k(T),\ \forall\ T\in\maT_n\}.
\end{eqnarray}
The subspace of mean zero functions is 
\begin{eqnarray*}
S_n^k=\{p\in P^k(\Omega),\ \int_\Omega prdrdz=0\}.
\end{eqnarray*}
Let $\mathbf V_n^k\in H^1_{-, 0}(\Omega)\times H^1_{+, 0}(\Omega)$ be the space with the boundary condition
\begin{eqnarray*}
\mathbf V_n^k=\{\bv=(v_r, v_z)\in [P^k(\Omega)]^2,\ v_r|_{\pa\Omega}=0,\ v_z|_\Gamma=0\}.
\end{eqnarray*}
Then, the Taylor-Hood finite element approximation for equation \eqref{eqn.5} is: For $k\geq 1$, find $(\bu_n, p_n)\in \mathbf V_n^{k+1}\times S^k_n$, such that for any $(\bv_n, q_n)\in \mathbf V^{k+1}_n\times S^k_n$,
\begin{eqnarray}\label{eqn.35}
\left\{\begin{array}{ll}
a(\bu_n, \bv_n)+ b(\bv_n,   p_n) =\int_{\Omega}\bbf\cdot\bv_n\\
 b(\bu_n,  q_n)=0,
\end{array}\right.
\end{eqnarray}
where $a(\cdot,\cdot)$ and $b(\cdot,\cdot)$ have the same formulation as in \eqref{eqn.4}, but act on  $\mathbf V^{k+1}_n\times S^k_n$. Under mild assumptions on the triangulation $\maT_n$ \cite{LL11}, the Taylor-Hood approximation satisfies the LBB inf-sup condition 
\begin{eqnarray*}
\sup_{0\neq \bv_n\in \mathbf V_n^{k+1}}\frac{b(\bv_n,  q_n)}{\|\bv_n\|_{H^1_{-}(\Omega)\times H^1_+(\Omega)}}\geq C\|q_n\|_{L^2_1(\Omega)},\quad \forall \ q_n\in S^{k}_n.
\end{eqnarray*}
Therefore, the finite element approximation is comparable to the best approximation from the space $\mathbf V^{k+1}_n\times S^k_n$,
\begin{eqnarray}
&&\|\bu-\bu_n\|_{H^1_{-}(\Omega)\times H^1_+(\Omega)}+\|p-p_n\|_{L^2_1(\Omega)}\nonumber\\
&&\leq C(\inf_{\bv_n\in \mathbf V^{k+1}_n}\|\bu-\bv_n\|_{H^1_{-}(\Omega)\times H^1_+(\Omega)}+\inf_{q_n\in S^k_n}\|p-q_n\|_{L^2_1(\Omega)}).\label{eqn.infapprox}
\end{eqnarray}

Recall the part of the boundary $\Gamma_0:=\pa\Omega\cap\{r=0\}$. For the local approximation property of the finite element solution $(\bu_n, p_n)$, we first recall the following interpolation operators from \cite{LL11}. 

For every node $x_i\in\pa\Omega$, we associated it to an edge $e(x_i)$ so that $x_i \in e(x_i)$. We require that $e(x_i)\cap\Gamma_0=\emptyset$  unless $x_i\in \Gamma_0$ and $e(x_i)\subset \Gamma$ if $x_i\in \Gamma$. Let $T(x_i)$ be a triangle containing $e(x_i)$, such that $T(x_i)\cap\Gamma_0=\emptyset$ if $e(x_i)\cap\Gamma_0=\emptyset$. Then, we define the local operator $\pi_i:H^1_+(T(x_i))\rightarrow \maP^{k}(e(x_i))$ by,
\begin{equation*}
\int_{e(x_i)} (\pi_i v)\psi r drdz = \int_{e(x_i)} v\psi r drdz,\quad \forall\ v \in H^1_+(T(x_i)), \, \forall\ \psi \in \mathcal{P}^{k}(e(x_i)).
\end{equation*}
For a node $x_i\notin\pa\Omega$, we associate it with a triangle $T(x_i)$, such that $x_i \in T(x_i)$ and define $\pi_i : H^1_+(T(x_i)) \rightarrow \mathcal{P}^{k}(T(x_i))$ by, 
\begin{equation*}
\int_{T(x_i)} (\pi_i v) \psi r dr dz = \int_{T(x_i)} v\psi r dr dz, \quad \forall \ v \in H^1_+(T(x_i)), \quad \forall \ \psi\in \mathcal{P}^{k}(T(x_i)).
\end{equation*}
Let $\phi_i$ be the usual finite element basis function at $x_i$. The interpolation operator $\Pi_{k, n}^+ : H^1_{+}(\Omega) \rightarrow P^{k}(\Omega)$ is 
\begin{equation}\label{interpo1}
\Pi_{k, n}^+ v := \sum_{i} (\pi_i v)(x_i) \phi_i, \quad \forall\ v \in H_+^1(\Omega).
\end{equation}
In addition, another operator $\Pi^-_{k, n}:H^1_-(\Omega)\rightarrow P^{k}(\Omega)$ was introduced for functions in $H^1_-(\Omega)$ as follows. For a node $x_i\in \Gamma_0$, we choose an edge $e(x_i)$ containing $x_i$ such that $e(x_i)\subset \Gamma_0$. Let $T(x_i) \in \mathcal{T}_n$ be a triangle that contains $e(x_i)$. We define the local operator $\pi_{i,r} : H_-^1(T(x_i)) \rightarrow \mathcal{P}^{k}(e(x_i))$ by
\begin{equation*}
\int_{e(x_i)}(\pi_{i,r}v) \psi drdz = \int_{e(x_i)} v\psi  drdz, \quad \forall \ v \in H^1_-(T(x_i)), \quad \forall \ \psi \in \mathcal{P}^{k}(e(x_i)). 
\end{equation*}
Then, the interpolation operator $\Pi_{k, n}^- : H^1_{-}(\Omega) \mapsto P^{k}(\Omega)$ is defined by
\begin{equation}\label{interpo2}
\Pi_{k, n}^- v = \sum_{\{i, x_i \notin \Gamma_0\}} \Pi_{k, n}^+ v(x_i) \phi_i + \sum_{\{i, x_i\in\Gamma_0\}} \pi_{i,r} v(x_i)\phi_i.
\end{equation}

With a weighted trace estimate. it has been shown that the interpolation operators $\Pi_{k, n}^+ : H^1_{+}(\Omega) \rightarrow P^{k}(\Omega)$ and $\Pi_{k, n}^- : H^1_{-}(\Omega) \rightarrow P^{k}(\Omega)$ are well-defined and preserve the zero boundary conditions  for functions in $H^1_{+, 0}(\Omega)$ and in $H^1_{-, 0}(\Omega)$, respectively. 
 The interpolations are also invariant for  functions in $P^k(\Omega)$. Let $T$ be a triangle in the triangulation $\maT_n$ and $U_T$ be the union of triangles intersecting $T$.  We have (Lemmas A.6 and A.7 in \cite{LL11})
\begin{eqnarray}
&\|\Pi_{{k, n}}^+ v\|_{L^2_1(T)}\leq C (h_T|v|_{H^{1}_1(U_T)}+\|v\|_{L^2_1(U_T)}),\label{newst1}\\
&|\Pi_{{k, n}}^+ v|_{H^1_1(T)}\leq C (|v|_{H^{1}_1(U_T)}+h_T^{-1}\|v\|_{L^2_1(U_T)}),\label{st2}\\
&\|r^{-1}\Pi_{{k, n}}^- v\|_{L^2_1(T)}\leq C \|v\|_{H^{1}_-(U_T)},\label{st3}\\
&\|\Pi_{{k, n}}^- v\|_{H^1_1(T)}\leq C (|v|_{H^{1}_1(U_T)}+h^{-1}_T\|v\|_{L^2_1(U_T)}),\label{st4}
\end{eqnarray}
where $h_T$ is the diameter of $U_T$.
Combining the stability and a Bramble-Hilbert Lemma in the weighted space $H^m_1(\Omega)$,   these interpolate operators consequently provide the following local approximation properties for $k\geq 1$ (Lemmas A.6 and A.8 in \cite{LL11}),
\begin{eqnarray}
&\|v-\Pi_{{k, n}}^+ v\|_{L^2_1(T)}\leq C h_T^{k+1}\|v\|_{H^{k+1}_1(U_T)}, \quad \forall \ v\in H^{k+1}_1(U_T)\label{36}\\
& |v-\Pi_{{k, n}}^+ v|_{H^1_+(T)}\leq C h_T^{k}\|v\|_{H^{k+1}_1(U_T)}, \quad \forall \ v\in H^{k+1}_1(U_T)\label{37}\\
&\|v - \Pi^-_{k, n} v\|_{H^1_-(T)} \leq  C h_T^{k}\| v\|_{H^{k+1}_1(U_T)}, \quad \forall \ v\in H^{k+1}_1(U_T)\cap H^1_-(U_T).\label{38}
\end{eqnarray}

The operators in \eqref{interpo1} and \eqref{interpo2} will be used for the approximation of the velocity. We will also need the following simpler interpolation operator from \cite{BBD06} for the approximation of the pressure.  

For each node $x_i$, we associate it with a triangle $T(x_i)$, such that $x_i\in T(x_i)$ and define $\pi^0_i$ as the $L^2_1(T(x_i))$ orthogonal projection of $v\in L^2_1(T(x_i))$ onto $\mathcal P^k(T(x_i))$:
\begin{eqnarray*}
\int_{T(x_i)} (\pi_i^0 v) \psi r dr dz = \int_{T(x_i)} v\psi r dr dz, \quad \forall \ v \in L^2_1(T(x_i)), \quad \forall \ \psi\in \mathcal{P}^{k}(T(x_i)).
\end{eqnarray*}
Then, we define
\begin{eqnarray}\label{interpo3}
\Pi_{k, n}v=\sum_{i}(\pi^0_iv)(x_i)\phi_i.
\end{eqnarray}
Using the same notation $U_T$ and $h_T$, the interpolation operator is stable (Theorem 1 in \cite{BBD06})
\begin{eqnarray}\label{st1}
\|\Pi_{k, n}v\|_{L^2_1(T)}\leq C\|v\|_{L^2_1(U_T)}
\end{eqnarray}
and yields the following approximation property
\begin{eqnarray}\label{newest11}
\|v-\Pi_{k,n} v\|_{L^2_1(T)}\leq Ch_T^{k+1}|v|_{H^{k+1}_1(U_T)}.
\end{eqnarray}

Recall the solution of the axisymmetric Stokes equation \eqref{eqn.5} may lack the regularity required in these  local estimates. These results, however, will help in our construction of special finite element spaces to approximate the singular solutions.

\subsection{Approximation of singular solutions}

\begin{algorithm} (The $\kappa$-refinement). \label{def.kappa}
Let $\kappa \in (0, 1/2]$ and $\maT$
be a triangulation of $\Omega$ such that no two vertices of
$\Omega$ belong to the same triangle of $\maT$. Then the {\em$\kappa$-refinement} of $\maT$, denoted by $\kappa(\maT)$, is obtained
by dividing each edge $AB$ of $\maT$ in two parts as follows.
If neither $A$ nor $B$ is in the vertex of $\Omega$, then we divide $AB$
into two equal parts. Otherwise, if 
$A$ is in $\mathcal Q$, we divide $AB$ into $AD$ and $DB$
such that $|AD| = \kappa |AB|$. This will
divide each triangle of $\maT$ into four triangles (Figure \ref{fig.refine}). 
\end{algorithm}
\begin{figure}
\includegraphics[scale=0.25]{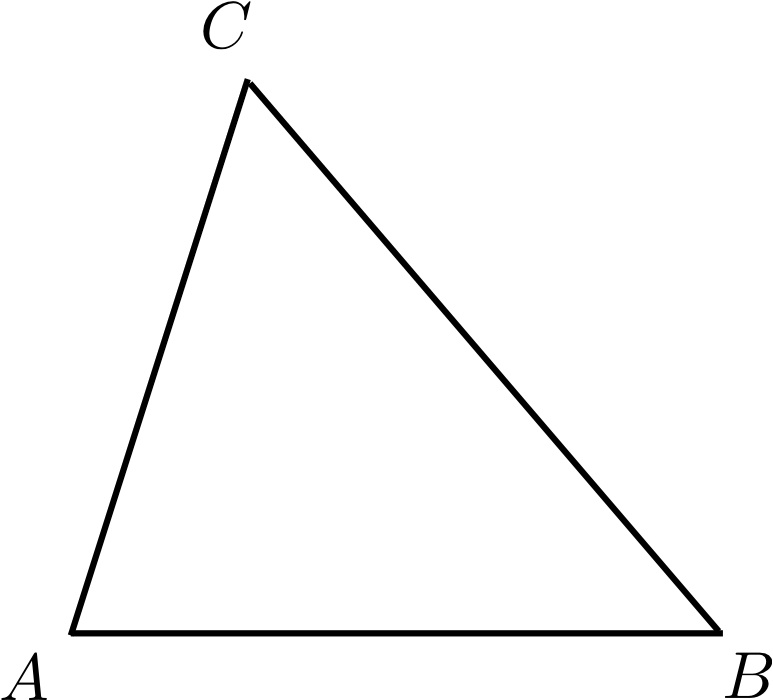}\hspace{0.5cm}\includegraphics[scale=0.25]{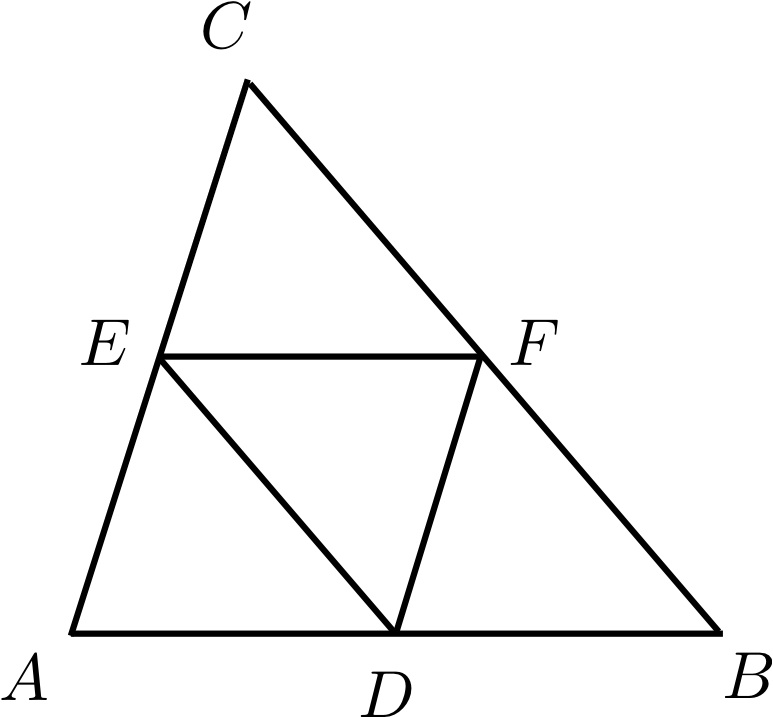}\hspace{0.5cm}\includegraphics[scale=0.25]{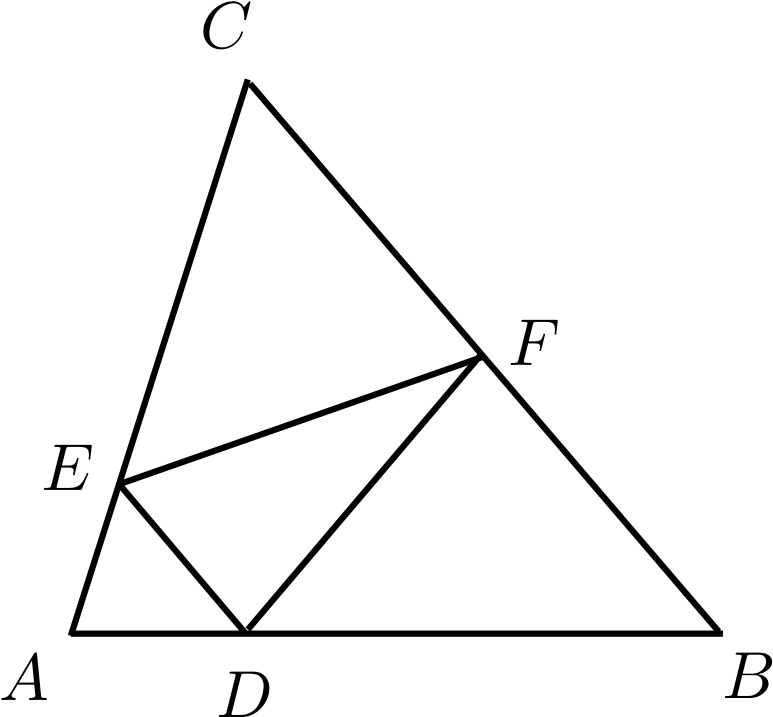}
\caption{An initial triangle $ABC$ (left); the mid-point decomposition if none of $A$, $B$, and $C$ belongs to $\mathcal Q$ (center);  the  $\kappa$-refinement if  $A \in \mathcal Q$, $\kappa=\frac{|AD|}{|AB|}=\frac{|AE|}{|AC|}$ (right).}\label{fig.refine}
\end{figure}

We now introduce the sequence of meshes. Recall  $L>0$
from Definition \ref{def.regular1}. 

\begin{definition} (The Graded Mesh). \label{def.4.4}
Suppose the initial mesh $\mathcal{T}_0$ of $\Omega$ is such that each
edge in the mesh has length $\leq L/2$ and each point in the vertex set $\mathcal Q$
is the vertex of a triangle in $\mathcal{T}_0$.  In addition,
we chose $\maT_0$ such that there is no triangle in $\maT_0$
that contains more than one point in $\mathcal Q$. Then we define
by induction $\maT_{j+1} = \kappa(\maT_j)$. 
\end{definition}

\begin{remark}
Definition \ref{def.4.4} gives a nested sequence of graded meshes by recursive applications of the $\kappa$-refinements. Note that the grading parameter $\kappa$ is fixed during the mesh generation. Then, the final triangulation contains shape-regular triangles where the class of shapes depends on the initial triangulation $\maT_0$ but not on the number of refinements. Therefore, the graded mesh satisfies the meshing requirement for the stability of the Taylor-Hood approximation of equation \eqref{eqn.5} \cite{LL11}. Since each triangle is decomposed into four small triangles for one refinement, the number of triangles in the triangulation $\maT_n$ is $\maO(4^n)$, and so is the dimension of the finite element space $\mathbf V^{k+1}_n\times S_n^k$. The $\kappa$-refinement generates triangles with different sizes adjusted for the singularity in the solution. Thus, the success of the graded mesh relies on the wise choice of the grading parameter $\kappa$, which we will elaborate on in this section.
\end{remark}

We need the following notation to carry out the analysis on graded meshes. Let $n$ be the number of $\kappa$-refinements of the domain $\Omega$. Thus, the final triangulation is $\maT_n$. Let $\mathbb T_{i, j}\subset \maT_j$, $j\leq n$, be the union of  triangles in $\maT_j$ that contain a vertex $Q_i\in\mathcal Q$ of $\Omega$. It can be seen that $\mathbb T_{i, j}\subset \mathbb T_{i, l}$ for $j\geq l$ and $\mathbb \cup_i T_{i, j}$ occupies the neighborhood of the vertex set $\mathcal Q$ in the triangulation $\maT_j$. Recall the regularity estimate for the solution and the parameter $\eta$ in Theorem \ref{thm.stokesreg}. We fix the grading parameter
\begin{eqnarray}\label{39}
\kappa:=\min(1/2, 2^{-\frac{k+1}{a}}), \qquad  {\rm{for}}\ 0<a<\eta,
\end{eqnarray} 
where $k\geq 1$ is the degree of piecewise polynomials in the Taylor-Hood finite element space $\mathbf V^{k+1}_n\times S^k_n$ associated with the triangulation $\maT_n$. Then, the error estimates for the Taylor-Hood approximation \eqref{eqn.35} are based on analysis on $\maT_n\backslash\cup_i\mathbb T_{i, 0}$, on $\cup_i\mathbb T_{i, j-1}\backslash \cup_i\mathbb T_{i, j}$, and on $\mathbb T_{i, n}$ summarized in the following lemmas. 
\begin{lemma}\label{lem.4.44}
For the space $P^k(\Omega)$ in \eqref{eqn.p} associated to  $\maT_n$ with $\kappa$ defined in \eqref{39}, let $U\subset \maT_n$ be the union of triangles that intersect $\maT_n\backslash\cup_i\mathbb T_{i, 0} $. Then,
\begin{eqnarray*}
\|u_r-\Pi^{-}_{k+1, n}u_{r}\|_{H^1_-(\maT_n\backslash\mathbb T_0 )}\leq C2^{-n(k+1)}\|u_r\|_{H^{k+2}_1(U)}\\
\|u_z-\Pi^{+}_{k+1, n}u_{z}\|_{H^1_+(\maT_n\backslash\mathbb T_0 )}\leq C2^{-n(k+1)}\|u_z\|_{H^{k+2}_1(U)}\\
\|p-\Pi_{k, n}p\|_{L^2_1(\maT_n\backslash\mathbb T_0 )}\leq C2^{-n(k+1)}\|p\|_{H^{k+1}_1(U)}.
\end{eqnarray*}
\end{lemma}
\begin{proof}
Assume $U$ is away from the vertices of the domain (this is true when $n>2$). Then, based on Definition \ref{def.4.4}, the mesh size on $U$ is $\maO(2^{-n})$. Summing up the estimates in \eqref{36}, \eqref{37},  \eqref{38}, and \eqref{newest11} completes the proof.
\end{proof}

For the estimates on $\mathbb T_{i, 0}$, the union of initial triangles containing the vertex $Q_i$, we consider the new coordinate system that is a simple translation of the old $rz$-coordinate system, now with $Q_i$ at the origin. Then, for a subset $G\subset \mathbb T_{i, 0}$ and $0<\lambda<1$, we define the dilation of $G$ and of a function as follows
\begin{eqnarray*}
G_\lambda:=G/\lambda, \quad v_{\lambda}(r_\lambda, z_\lambda):=v(r, z), \quad {\rm{for}}\ (r_\lambda, z_\lambda)=(\lambda^{-1}r, \lambda^{-1}z)\in G_\lambda.
\end{eqnarray*}
Then,
\begin{lemma}\label{lem.dilation}
Suppose $G_\lambda\subset \maV_i$. Then, if $Q_i$ is on  the $z$-axis,
\begin{eqnarray*}
\|v_{\lambda}\|_{\maK^m_{a, 1}(G_\lambda)} = \lambda^{a-3/2}\|v\|_{\maK^m_{a, 1}(G)},\\
\|r^{-1}_\lambda v_{\lambda}\|_{L^2_1(G_\lambda)} = \lambda^{-1/2}\|r^{-1}v\|_{L^2_1(G)};
\end{eqnarray*}
 if $Q_i$ is not on the $z$-axis, 
 \begin{eqnarray*}
 C\lambda^{a-1}\|v\|_{\maK^m_{a, 1}(G)}\leq\|v_{\lambda}\|_{\maK^m_{a, r}(G_\lambda)} \leq C\lambda^{a-1}\|v\|_{\maK^m_{a, 1}(G)}.
 \end{eqnarray*}
\end{lemma}
\begin{proof}
 Note that on both $G_\lambda\subset \maV_i$ and $G\subset
\maV_i$, $\vartheta(r, z)$ is equal to the distance from $(r, z)$ to $Q_i$,
therefore $\vartheta(r_\lambda, z_\lambda)=\lambda^{-1}\vartheta(r, z)$. Then, if $Q_i\in \{r=0\}$, 
\begin{eqnarray*}
&&	\|v_\lambda\|^2_{\maK^m_{a, 1}(G_\lambda)}  =  \sum_{j+k\leq
	m}\int_{G_\lambda} |\vartheta^{j+k-a}(r_\lambda, z_\lambda)\partial_{r_\lambda}^j\partial^k_{z_\lambda}
	v_\lambda(r_\lambda, z_\lambda)|^2 r_\lambda dr_\lambda dz_\lambda\\ 
	&& = \sum_{j+k\leq m} \int_{G}
	|\lambda^{a-j-k}\vartheta^{j+k-a}(r, z) \lambda^{j+k} \partial_r^{j}
	\partial^k_z v(r, z)|^2 \lambda^{-3}r drdz\\ && =
	\lambda^{2a-3}\sum_{j+k\leq m}\int_{G} |\vartheta^{j+k-a}(r,z)
	\partial_r^{j} \partial^k_z v(r, z)|^2 rdrdz = 
	\lambda^{2a-3}\|v\|^2_{\maK^m_{a, 1}(G)}.
\end{eqnarray*}
In addition, 
\begin{eqnarray*}
\|r^{-1}_\lambda v_{\lambda}\|_{L^2_1(G_\lambda)}^2=\int_{G_\lambda} |v_\lambda(r_\lambda, z_\lambda)|^2 r_\lambda^{-1} dr_\lambda dz_\lambda\\
=\lambda^{-1}\int_{G} |v(r, z)|^2 r^{-1} dr dz=\lambda^{-1/2}\|r^{-1}v\|_{L^2_1(G)}.
\end{eqnarray*}

On the other hand, if $Q_i\notin\{r=0\}$, we notice $A\leq r^{-1}\leq B$ on $\mathcal{V}_i$, for constants $A$ and $B$ depending on the domain $\Omega$. Therefore, we have,
$$
A \|v(r, z)\|^2_{\maK^m_{a, 1}(D)} \leq \sum_{j+k\leq
	m}\int_{D} |\vartheta^{j+k-a}(r, z)\partial_r^j\partial^k_z
	v(r, z)|^2 drdz\leq B\|v(r, z)\|^2_{\maK^m_{a, 1}(D)},$$
where $D\subset \maV_i$ is any subset of $\maV_i$. We thus have
\begin{eqnarray*}
	&&\|v_\lambda(r_\lambda, z_\lambda)\|^2_{\maK^m_{a, 1}(G_\lambda)}  \leq  A^{-1}\sum_{j+k\leq
	m}\int_{G_\lambda} |\vartheta^{j+k-a}(r_\lambda, z_\lambda)\partial_{r_\lambda}^j\partial^k_{z_\lambda}
	v_\lambda(r_\lambda, z_\lambda)|^2 dr_\lambda dz_\lambda\\ && = A^{-1}\sum_{j+k\leq m} \int_{G}
	|\lambda^{a-j-k}\vartheta^{j+k-a}(r, z) \lambda^{j+k} \partial_r^{j}
	\partial^k_z v(r, z)|^2 \lambda^{-2}drdz\\ && =
	A^{-1}\lambda^{2a-2}\sum_{j+k\leq m}\int_{G} |\vartheta^{j+k-a}(r,z)
	\partial_r^{j} \partial^k_z v(r, z)|^2drdz \leq A^{-1}B
	\lambda^{2a-2}\|v\|^2_{\maK^m_{a, 1}(G)}.
\end{eqnarray*}
We note the inequality in the opposite direction can be justified with the same process, which completes the proof.
\end{proof}

We are ready to give estimates on the region $\mathbb T_{i, j-1}\backslash \mathbb T_{i, j}$. From now on, we assume the constant $a$ in the sub-index of the space is always non-negative.

\begin{lemma}\label{proposition.4.3}
For the space $P^k(\Omega)$ in \eqref{eqn.p} associated to  $\maT_n$ with $\kappa$ defined in \eqref{39}, let $U\subset \maT_n$ be the union of triangles that intersect $G:=\mathbb T_{i, j-1}\backslash\mathbb T_{i, j} $. Let $h$ be the mesh size on $U$ and $\xi=\sup_{x\in G} \vartheta(x)$. Then, for $v\in H^1_-(U)\cap\maK^{k+1}_{a+1, 1}(U)$,
\begin{eqnarray}
     &\|r^{-1}(v-\Pi^-_{k, n} v)\|_{L^2_{1}(G)}+\|v-\Pi^-_{k, n} v\|_{H^1_{1}(G)}\leq C\xi^{a}(h/\xi)^{k}\|v\|_{\maK^{k+1}_{a+1, 1}(U)};\label{es1}
     \end{eqnarray}
     and for $v\in \maK^{k+1}_{a+1,1}(U)$,
     \begin{eqnarray}
    & \|v-\Pi^+_{k, n} v\|_{H^1_{1}(G)}\leq
    C\xi^{a}(h/\xi)^{k}\|v\|_{\maK^{k+1}_{a+1, 1}(U)},\label{es2}\\
    & \|v-\Pi_{k, n} v\|_{L^2_{1}(G)}\leq
    C\xi^{a}(h/\xi)^{k+1}\|v\|_{\maK^{k+1}_{a, 1}(U)}.\label{es3}
\end{eqnarray}
\end{lemma}
\begin{proof}
Recall the new coordinate system with $Q_i$ as the origin. Let $G_\lambda=\lambda^{-1}G$. Recall the dilation function $v_\lambda(r_\lambda, z_\lambda)=v( r, z)$. Note
that 
$(\Pi^{+}_{k, n} v)_{\lambda}=\Pi_{k,n}^+ (v_{\lambda})$ and $(\Pi^{-}_{k, n} v)_{\lambda}=\Pi_{k,n}^- (v_{\lambda})$. 
Then, we  choose  $\lambda=2\xi/L$, such that $G_\lambda\subset \maV_i$. 

If $Q_i$ is on the $z$-axis, by Lemma \ref{lem.dilation}, the definitions of the weighted spaces,  and \eqref{38}, we have
\begin{eqnarray*}
&&\|r^{-1}(v-\Pi^-_{k, n} v)\|_{L^2_{1}(G)}+\|v-\Pi^-_{k, n} v\|_{H^1_{1}(G)}\\
&&\leq C\|r^{-1}(v-\Pi^-_{k, n} v)\|_{L^2_{1}(G)}+\|v-\Pi^-_{k, n} v\|_{\maK^1_{1, 1}(G)}\\
&&=\lambda^{1/2}(\|r^{-1}_\lambda(v_\lambda-\Pi^-_{k, n} (v_\lambda)\|_{L^2_{1}(G_\lambda)}+\|v_\lambda-\Pi^-_{k, n}( v_\lambda)\|_{\maK^1_{1, 1}(G_\lambda)})\\
&&\leq C\lambda^{1/2}(\|r^{-1}_\lambda(v_\lambda-\Pi^-_{k, n} (v_\lambda)\|_{L^2_{1}(G_\lambda)}+\|v_\lambda-\Pi^-_{k, n}( v_\lambda)\|_{H^1_{1}(G_\lambda)})\\
&&\leq C\lambda^{1/2}(h/\lambda)^k\|v_\lambda\|_{H^{k+1}_1(U_\lambda)} \leq C\lambda^{1/2}(h/\lambda)^k\|v_\lambda\|_{\maK^{k+1}_{1, 1}(U_\lambda)}\\
&&\leq C(h/\xi)^k\|v\|_{\maK^{k+1}_{1,1}(U)}\leq C\xi^a(h/\xi)^k\|v\|_{\maK^{k+1}_{a+1,1}(U)}.
\end{eqnarray*}
If  $Q_i$ is not on the $z$-axis, the proof is similar. With the corresponding estimate in Lemma \ref{lem.dilation}, the definitions of the weighted spaces, and \eqref{38}, we have
\begin{eqnarray*}
&&\|r^{-1}(v-\Pi^-_{k, n} v)\|_{L^2_{1}(G)}+\|v-\Pi^-_{k, n} v\|_{H^1_{1}(G)}\\
&&\leq C\|r^{-1}(v-\Pi^-_{k, n} v)\|_{L^2_{1}(G)}+\|v-\Pi^-_{k, n} v\|_{\maK^1_{1, 1}(G)}\leq C\|v-\Pi^-_{k, n} v\|_{\maK^1_{1, 1}(G)}\\
&&\leq C\|v_\lambda-\Pi^-_{k, n}( v_{\lambda})\|_{\maK^1_{1, 1}(G_\lambda)}\leq C\|v_\lambda-\Pi^-_{k, n}(v_{\lambda})\|_{H^1_{1}(G_\lambda)}\leq  C(h/\lambda)^k\|v_\lambda\|_{H^{k+1}_1(U_\lambda)}\\
&&\leq C(h/\lambda)^{k}\|v_\lambda\|_{\maK^{k+1}_{1, 1}(U_\lambda)} \leq C(h/\xi)^k\|v\|_{\maK^{k+1}_{1, 1}(U)}\leq
C \xi^{a}(h/\xi)\|v\|_{\maK^{k+1}_{a+1, 1}(U)}.
\end{eqnarray*}
Thus, the estimate in \eqref{es1} is proved.

The estimates in \eqref{es2} and \eqref{es3} can be shown similarly using Lemma \ref{lem.dilation}, the definitions of the weighted spaces,  \eqref{36},  \eqref{37}, and \eqref{newest11}.
\end{proof}

\begin{lemma}\label{prop.4.5}
For the space $P^k(\Omega)$ in \eqref{eqn.p} associated to  $\maT_n$ with $\kappa$ defined in \eqref{39}, let $U\subset \maT_n$ be the union of triangles that intersect $G:=\mathbb T_{i, j-1}\backslash\mathbb T_{i, j} $. Then,
\begin{eqnarray*}
&\|r^{-1}(u_r-\Pi^{-}_{k+1, n}u_{r})\|_{L^2_1(G)}+\|u_r-\Pi^{-}_{k+1, n}u_{r}\|_{H^1_{1}(G)}\leq C2^{-n(k+1)}\|u_r\|_{\maK^{k+2}_{a+1, 1}(U)}\\
&\|u_z-\Pi^{+}_{k+1, n}u_{z}\|_{H^1_{1}(G)}\leq C2^{-n(k+1)}\|u_z\|_{\maK^{k+2}_{a+1, 1}(U)}\\
&\|p-\Pi_{k, n}p\|_{L^2_1(G )}\leq C2^{-n(k+1)}\|p\|_{\maK^{k+1}_{a, 1}(U)}.
\end{eqnarray*}
\end{lemma}

\begin{proof}
Definition \ref{def.4.4} shows that the mesh on $\mathbb T_{i, j-1}\backslash\mathbb T_{i, j}$ and also on $U$ has the size $\maO( \kappa^{j-1}2^{j-1-n})$. Using the notation of Lemma \ref{proposition.4.3}, we
have $\xi=\mathcal O(\kappa^{j-1})$ on $\mathbb T_{i, {j-1}} \backslash\mathbb T_{i, {j}}$. Therefore, using Lemma \ref{proposition.4.3}, we have
\begin{eqnarray*}
&&\|r^{-1}(u_r-\Pi^{-}_{k+1, n}u_{r})\|_{L^2_1(G)}+\|u_r-\Pi^{-}_{k+1, n}u_{r}\|_{\maK^1_{1, 1}(G)}\\
&&\leq C\kappa^{(j-1)a}2^{(j-1-n)(k+1)}\|u_r\|_{\maK^{k+2}_{a+1, 1}(U)} \leq
C2^{-(j-1)(k+1)}2^{(j-1-n)(k+1)}\|u_r\|_{\maK^{k+2}_{a+1, 1}(U)}\\ 
&&= C2^{-n(k+1)}\|u_r\|_{\maK^{k+2}_{a+1, 1}(U)}.
\end{eqnarray*}
Then, we have proved the first estimate in this lemma. The last two estimates can be proved similarly by Lemma \ref{proposition.4.3} and the observation on the mesh size for the regions $G$ and $U$.
\end{proof}

The following lemma gives the error bounds on the last patch $\mathbb T_{i, n}$ of triangles that have the vertex $Q_i$ as the common node.
\begin{lemma}\label{prop.4.6}
For the space $P^k(\Omega)$ in \eqref{eqn.p} associated to  $\maT_n$ with $\kappa$ defined in \eqref{39}, let $U\subset \maT_n$ be the union of triangles that intersect $\mathbb T_{i, n} $. Then,
\begin{eqnarray*}
&\|r^{-1}(u_r-\Pi^{-}_{k+1, n}u_{r})\|_{L^2_1(\mathbb T_{i, n} )}+\|u_r-\Pi^{-}_{1, n}u_{r}\|_{H^1_{1}(\mathbb T_{i, n} )}\\
&\leq C2^{-n(k+1)}(\|u_r\|_{\maK^{1}_{a+1, 1}(U)}+\|\vartheta^{-a}r^{-1}u_r\|_{L^2_1(U)})\\
&\|u_z-\Pi^{+}_{k+1, n}u_{z}\|_{H^1_{1}(\mathbb T_{i, n} )}\leq C2^{-n(k+1)}\|u_z\|_{\maK^{1}_{a+1, 1}(U)}\\
&\|p-\Pi_{k, n}p\|_{L^2_1(\mathbb T_{i, n} )}\leq C2^{-n(k+1)}\|p\|_{\maK^{0}_{a, 1}(U)}.
\end{eqnarray*}
\end{lemma}

\begin{proof}
Definition \ref{def.4.4} shows that the mesh on $U$ has the size $\maO( \kappa^{n})$. By the stability results in \eqref{st3} and \eqref{st4}, we have 
\begin{eqnarray*}
&\|r^{-1}(u_r-\Pi^{-}_{k+1, n}u_{r})\|_{L^2_1(\mathbb T_{i, n})}+\|u_r-\Pi^{-}_{k+1, n}u_{r}\|_{H^1_{1}(\mathbb T_{i, n})}\\
&\leq C(\|r^{-1}u_r\|_{L^2_1(U)}+\kappa^{-n}\| u_r\|_{L^2_1(U)}+| u_r|_{H^1_1(U)})\\
&\leq C(\kappa^{na}\|\vartheta^{-a}r^{-1}u_r\|_{L^2_1(U)}+\kappa^{-n}\kappa^{n(1+a)}\| u_r\|_{\maK^0_{a+1,1}(U)}+\kappa^{na}| u_r|_{\maK^1_{a+1, 1}(U)})\\
&\leq C2^{-n{(k+1)}}(\|\vartheta^{-a}r^{-1}u_r\|_{L^2_1(U)}+\| u_r\|_{\maK^1_{a+1, 1}(U)}).
\end{eqnarray*}
Then, we have proved the first estimate in this lemma. The last two estimates can be proved similarly by the stability results in \eqref{newst1}, \eqref{st2}, and \eqref{st1}.
\end{proof}

\begin{theorem}\label{thm.main2}
Recall the parameter $\eta$ from Theorem \ref{thm.stokesreg}. For the finite element space $\mathbf V^{k+1}_n\times S^k_n$,  $k\geq 1$, on  $\maT_n$ with $\kappa$ defined in \eqref{39}, let $(\bu_n,p_n)\in \mathbf V^{k+1}_n\times S^k_n$ be the Taylor-Hood finite element approximation of the axisymmetric Stokes equation in \eqref{eqn.35}. Let $N:={\rm{dim}}(\mathbf V^{k+1}_n\times S^k_n)$ be the dimension of the finite element space. Then, if $\bbf\in \maK^{k}_{a-1, -}(\Omega)\times\maK^{k}_{a-1, +}(\Omega)$, for $0<a<\eta$,
\begin{eqnarray*}
&\|\bu-\bu_n\|_{H^1_-(\Omega)\times H^1_+(\Omega)}+\|p-p_n\|_{L^2_{1}(\Omega )}\leq CN^{-(k+1)/2}\|\bbf\|_{\maK^{k}_{a-1, -}(\Omega)\times\maK^{k}_{a-1, +}(\Omega)}.
\end{eqnarray*}
\end{theorem}
\begin{proof}
Since the interpolation operators $\Pi^{+}_{k, n}$ and $\Pi^{-}_{k, n}$ preserves the zero boundary condition for functions in $H^1_{+, 0}(\Omega)$ and in $H^1_{-, 0}(\Omega)$, respectively, we have 
\begin{eqnarray*}
(\Pi^{-}_{k+1, n}u_r, \Pi^{+}_{k+1, n}u_z)\in \mathbf V^{k+1}_n.
\end{eqnarray*}
Let $q_n\in P^k(\Omega)$ be such that
\begin{eqnarray}\label{qqq}
\int_\Omega q_nv_nrdrdz=\int_\Omega pv_nrdrdz,\quad \forall\ v_n\in P^k(\Omega).
\end{eqnarray}
Choosing $v_n=1$, we see that $q_n\in S^k_n$. Note that summing up the estimates for $p-\Pi_{n, k}p$ in Lemmas \ref{lem.4.44}, \ref{prop.4.5}, and \ref{prop.4.6}, we have
\begin{eqnarray}\label{perror}
\inf_{\chi_n\in P^k(\Omega)}\|p-\chi_n\|_{L^2_1(\Omega)}\leq \|p-\Pi_{k, n}p\|_{L^2_1(\Omega)}\leq C2^{-n(k+1)}\|p\|_{\mathcal K^{k+1}_{a, 1}(\Omega)}.
\end{eqnarray}
The infimum is achieved by the $L^2_1(\Omega)$ projection of $p$ onto $P^k(\Omega)$, which is $q_n$ in \eqref{qqq}.
Therefore, by \eqref{eqn.infapprox} and \eqref{perror}, we have
\begin{eqnarray*}
&&\|\bu-\bu_n\|_{H^1_{-}(\Omega)\times H^1_+(\Omega)}+\|p-p_n\|_{L^2_1(\Omega)}\nonumber\\
&&\leq C(\inf_{\bv_n\in \mathbf V^{k+1}_n}\|\bu-\bv_n\|_{H^1_{-}(\Omega)\times H^1_+(\Omega)}+\inf_{\chi_n\in S^k_n}\|p-\chi_n\|_{L^2_1(\Omega)})\\
&&\leq C(\|u_r-\Pi^{-}_{k+1, n}u_r\|_{H^1_{-}(\Omega)}+ \|u_z-\Pi^{+}_{k+1, n}u_z\|_{H^1_+(\Omega)}+\|p-q_n\|_{L^2_1(\Omega)})\\
&&\leq C(\|u_r-\Pi^{-}_{k+1, n}u_r\|_{H^1_{-}(\Omega)}+ \|u_z-\Pi^{+}_{k+1, n}u_z\|_{H^1_+(\Omega)}+\|p-\Pi_{k, n}p\|_{L^2_1(\Omega)}).
\end{eqnarray*}
Then, summing up the estimates in Lemmas \ref{lem.4.44}, \ref{prop.4.5}, and \ref{prop.4.6}, we have
\begin{eqnarray*}
&\|\bu-\bu_n\|_{H^1_{-}(\Omega)\times H^1_+(\Omega)}+\|p-p_n\|_{L^2_1(\Omega)}\nonumber\\
&\leq C2^{-n(k+1)}(\|\vartheta^{-a}r^{-1}u_r\|_{L^2_1(\Omega)}+\|\bu\|_{[\maK^{k+2}_{a+1, 1}(\Omega)]^2}+\|p\|_{\maK^{k+1}_{a, 1}(\Omega)}).
\end{eqnarray*}
Recall the dimension of the finite element space $N=\mathcal O(4^{n})$. By Theorem \ref{thm.stokesreg} we complete the proof by concluding
\begin{eqnarray*}
\|\bu-\bu_n\|_{H^1_{-}(\Omega)\times H^1_+(\Omega)}+\|p-p_n\|_{L^2_1(\Omega)}\leq CN^{-(k+1)/2}\|\bbf\|_{\maK^{k}_{a-1, -}(\Omega)\times\maK^{k}_{a-1, +}(\Omega)}.
\end{eqnarray*}
\end{proof}

\begin{remark}
Note that near the vertices, our refinement has similar properties to the ones
  in \cite{BNZ205, LN09, MR0478669, Raugel78}.  Regularity is a local property. Instead of using the same parameter $\eta$ for all the vertices of the domain, one can specify a different $\eta_i$ for a different vertex $Q_i$, depending on its location and the interior angle (see Lemmas \ref{lem.2dstokes} and \ref{lem.3dstokes} for the local characterization of $\eta$.). The global regularity estimate in Theorem \ref{thm.stokesreg} still holds if we replace $a$ and $\eta$ by $\ba=(a_i)$ and $\beeta=(\eta_i)$, respectively, where the space $\maK^m_{\bmu, 1}(\Omega)$ can be defined similarly as the space $\maK^m_{\mu, 1}(\Omega)$, but with the specific weight parameter $\mu_i$ (instead of the uniform parameter $\mu$) for the $i$th vertex (see also \cite{LMN10} for weighted spaces with vector indices.). This will increase the flexibility for the use of graded meshes with different grading parameters for different vertices.
\end{remark}

\section{Numerical Illustrations}\label{sec5} 

In this section, we present sample numerical results that confirm our theoretical analysis. In particular, we shall justify the use of graded meshes to recover  the optimal rate of convergence of the finite element approximation for singular solutions of the axisymmetric Stokes equation \eqref{eqn.5}, as predicted in Theorem \ref{thm.main2}. 

\subsection{Numerical experiments} Our numerical tests are implemented on two domains, corresponding to the singularities in  solutions away and on the $z$-axis, respectively. Recall that the determination (the value of $\eta$ in \eqref{eta}) of the optimal graded meshes is based on different criteria for these two cases. In both tests, we use the P2-P1 Taylor-Hood mixed formulation \eqref{eqn.35} and set $f_r=4r^{3/5}$, $f_z=8r^{3/5}\cos{z}$. Note that with this choice, $\bbf\in H^1_-(\Omega)\times H^1_+(\Omega)$ and $\bbf\in \maK^1_{a-1, -}(\Omega)\times \maK^1_{a-1, +}(\Omega)$ for any $0<a<3/2$. 

\begin{figure}
\centering
\hspace{1cm}\includegraphics[scale=0.18]{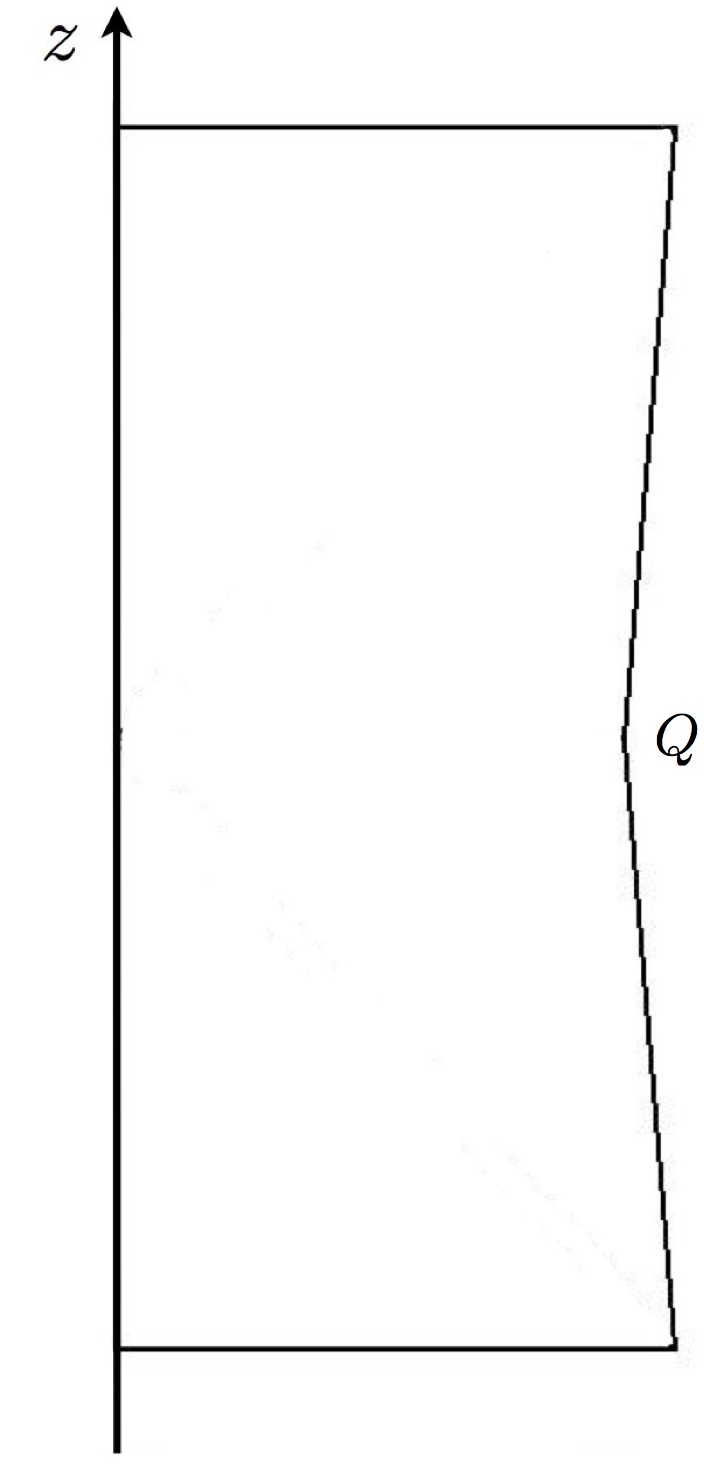}\hspace{4cm}\includegraphics[scale=0.16]{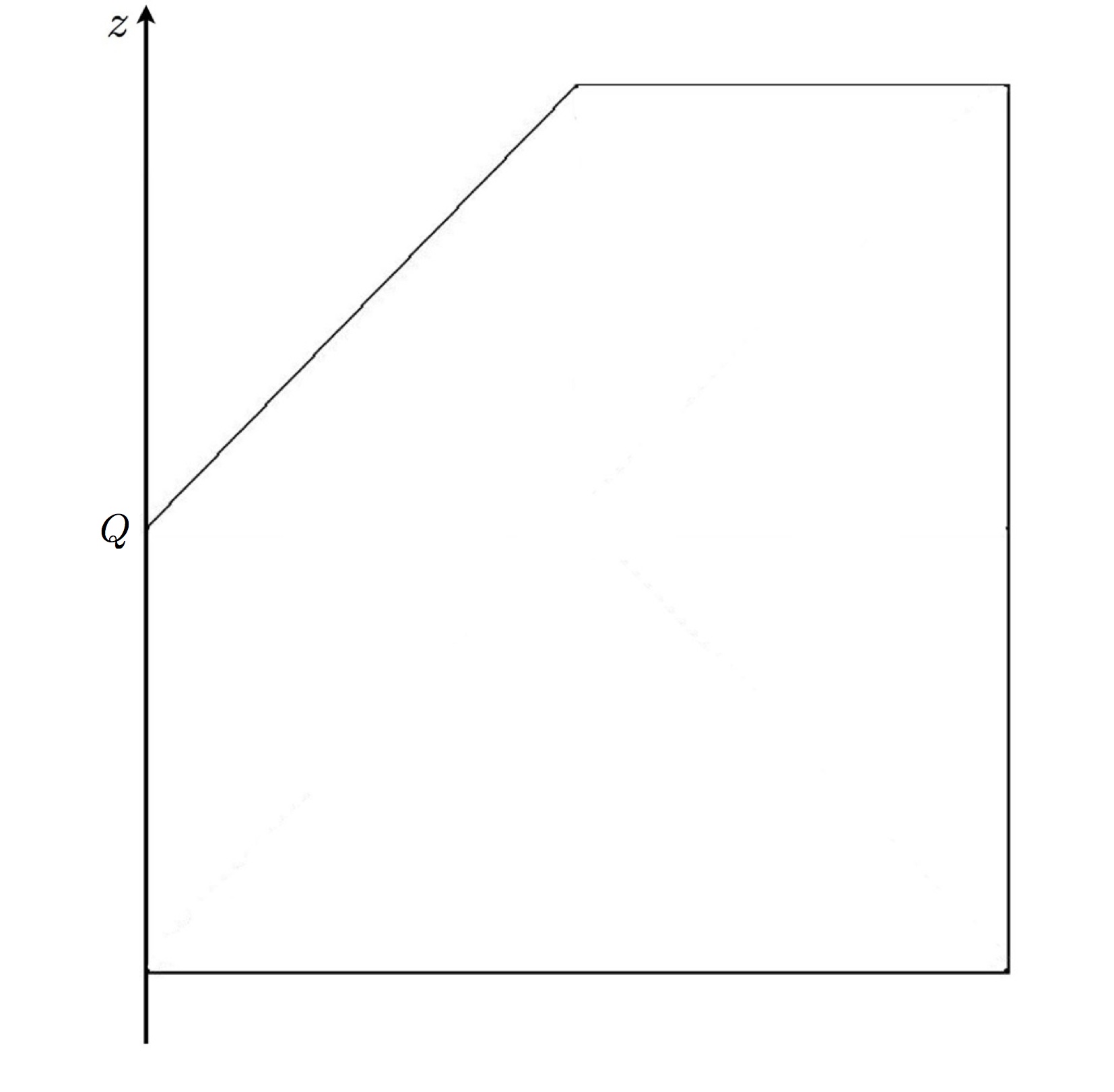}
\caption{Computational domains: $\Omega_1$ (left); $\Omega_2$ (right).}\label{fig1}
\end{figure}

\begin{figure}
\centering
\includegraphics[scale=0.1]{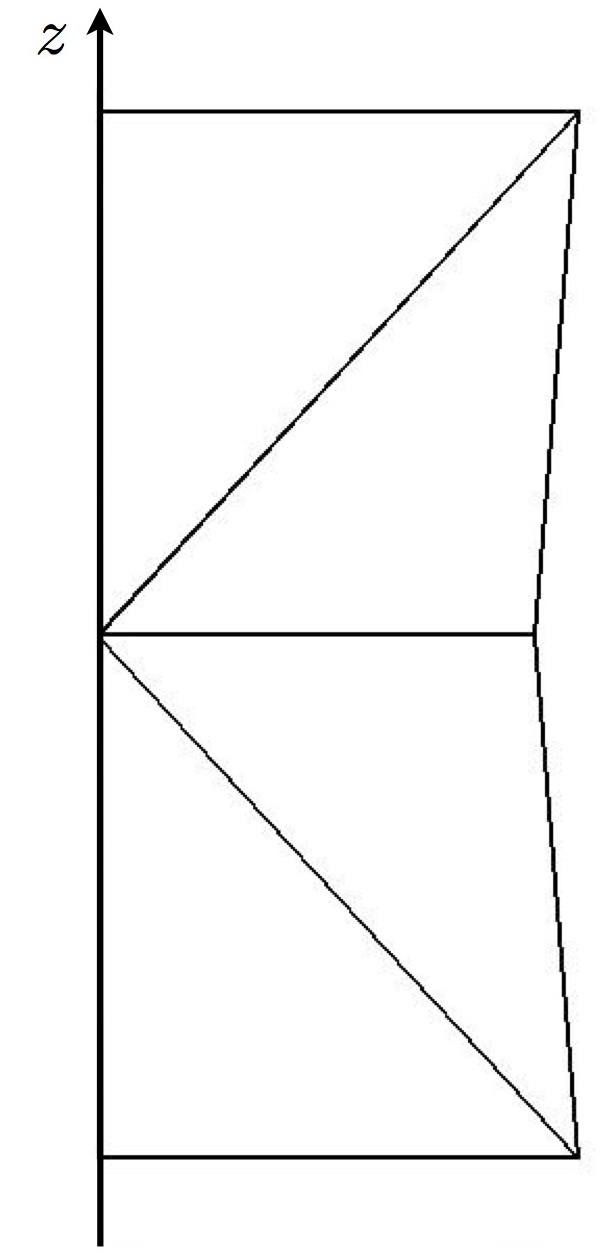}\hspace{2cm}\includegraphics[scale=0.1]{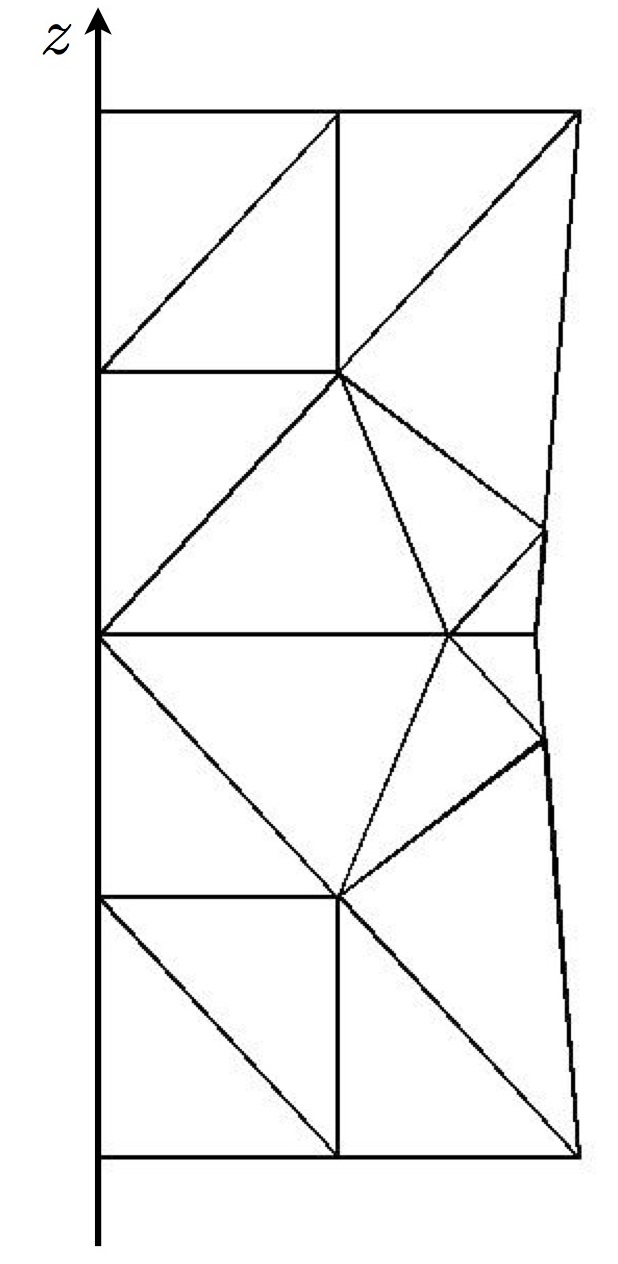}\hspace{2cm}\includegraphics[scale=0.1]{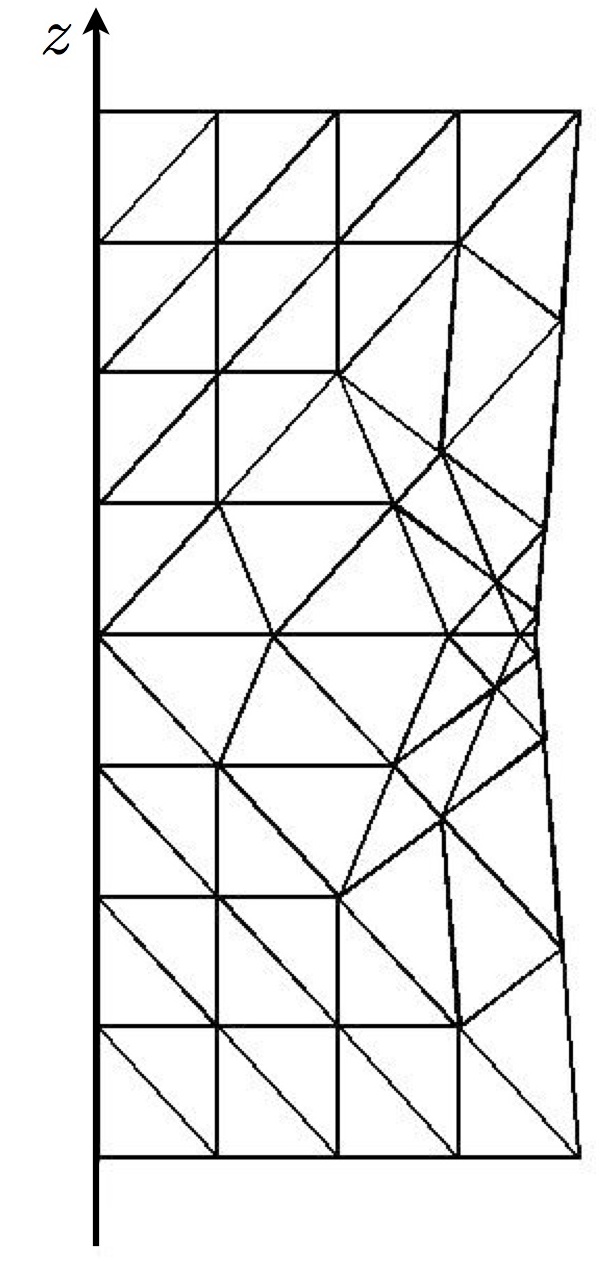}
\caption{Consecutive $\kappa$-refinements for $\Omega_1$ ($\kappa=0.2$): the initial triangulation $\maT_0$ (left); the  mesh after one refinement $\maT_1$ (center); the mesh after two refinements $\maT_2$ (right).}\label{fig10}
\end{figure}

\begin{figure}[htbp]
\centering
\includegraphics[scale=0.15]{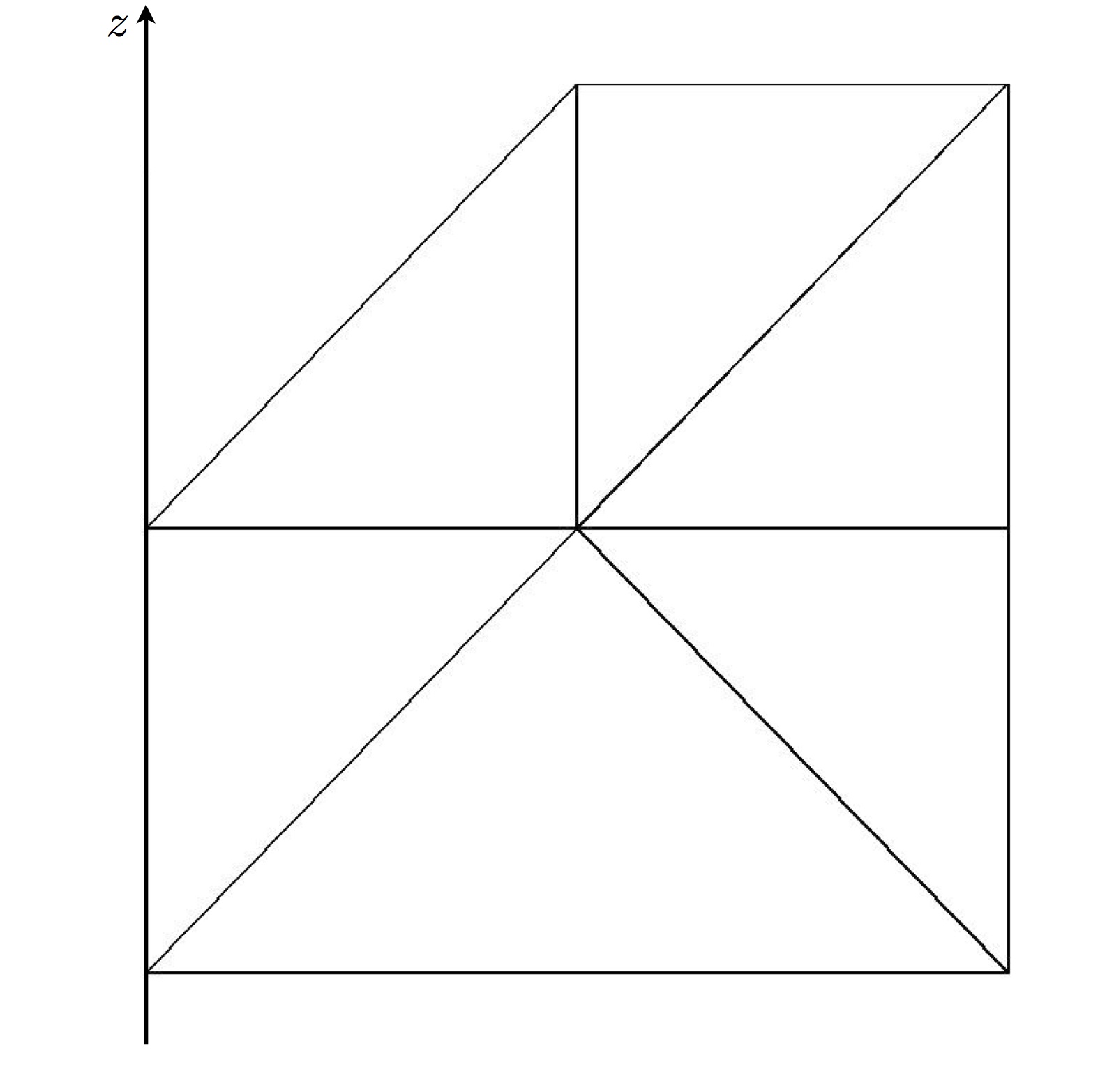}\hspace{0.0cm}\includegraphics[scale=0.15]{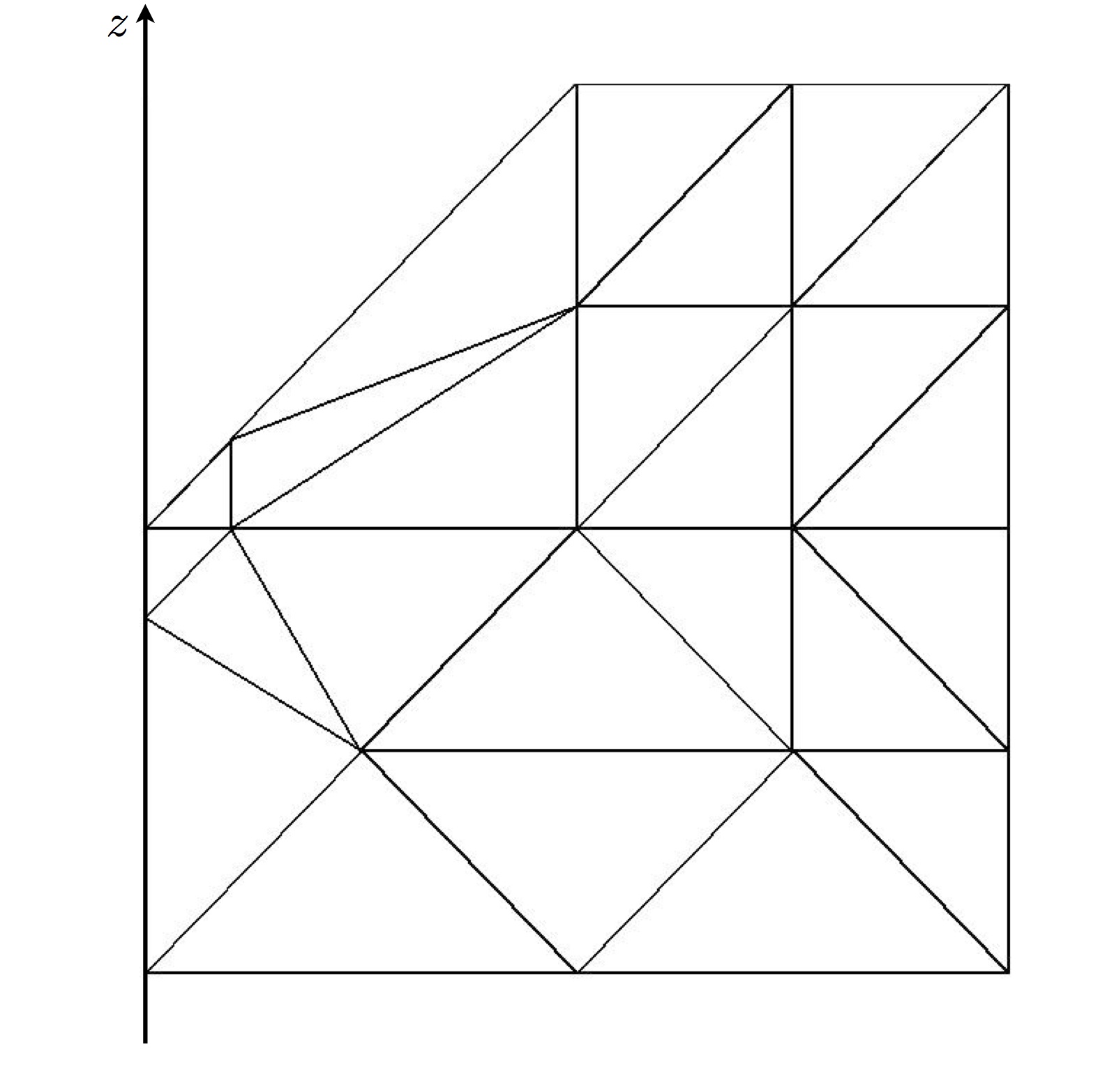}\hspace{0.0cm}\includegraphics[scale=0.15]{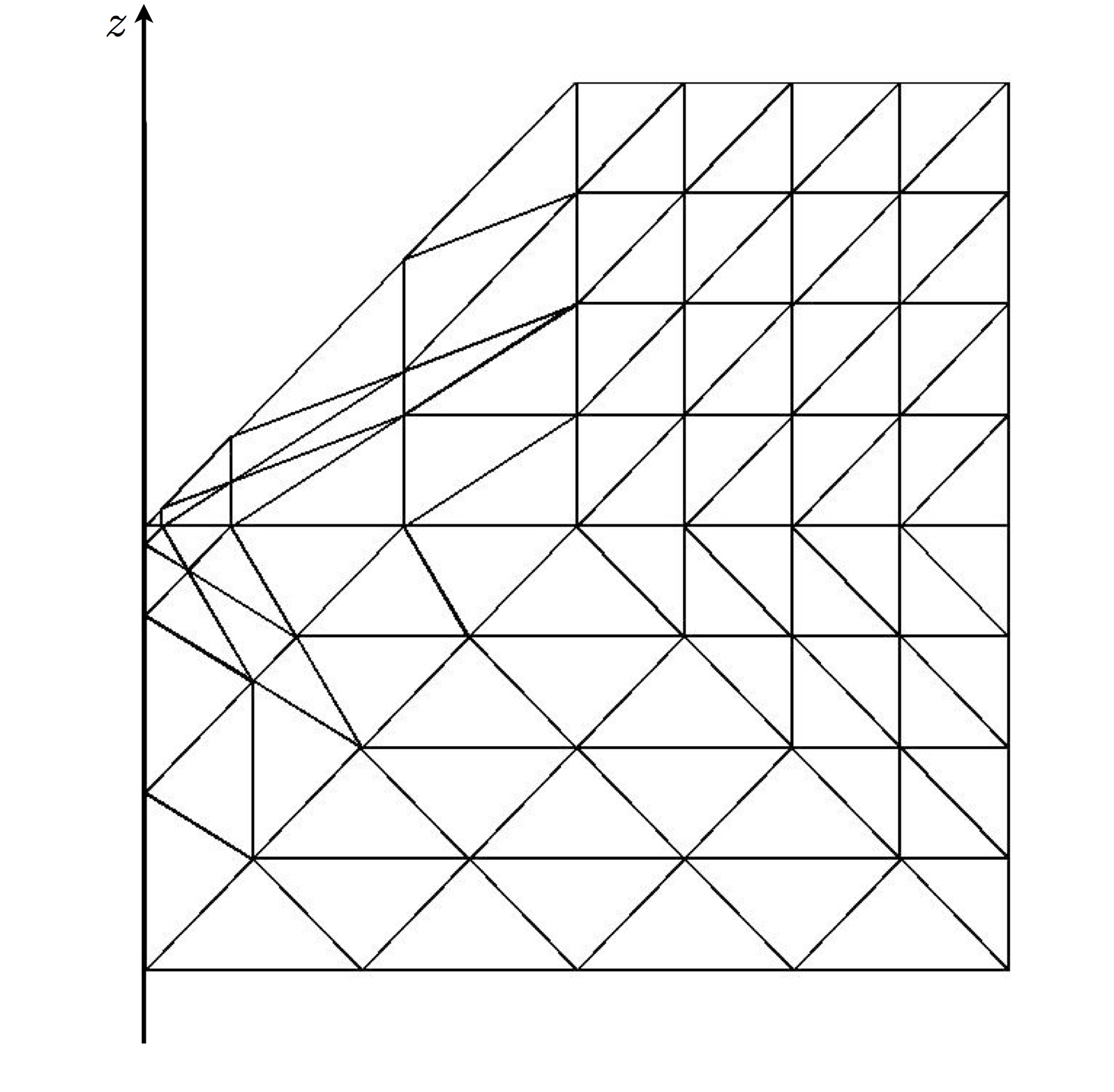}
\caption{Consecutive $\kappa$-refinements for $\Omega_2$ ($\kappa=0.2$): the initial triangulation $\maT_0$ (left); the  mesh after one refinement $\maT_1$ (center); the mesh after two refinements $\maT_2$ (right).}
\end{figure}

We first consider the axisymmetric Stokes equations  on a polygonal domain $\Omega_1$ (the first domain in Figure \ref{fig1}). The interior angle at the vertex $Q$ is $1.05\pi$ and other interior angles $\leq 0.5\pi$. It can be shown that the solution $(\bu, p)\in H^s_1$ near $Q$, for $s<3$ and $(\bu, p)\in H^3_1\times H^3_1\times H^2_1$ in other parts of the domain. In fact, based on the calculation for the eigenvalues of the operator pencil \cite{BR81}, $\eta\approx0.909$ for the vertex $Q$. Therefore, based on Theorem \ref{thm.main2}, the  graded mesh near $Q$ should have the parameter $\kappa<2^{-2/\eta}\approx0.212$ to recover the optimal rate of convergence for the P2-P1 element. Using the same initial triangulation $\maT_0$, We have tested the numerical errors and convergence rates between consecutive numerical solutions up to eight levels of graded refinements for different values of $\kappa$ near the vertex $Q$.

The results of these tests for $\Omega_1$ are listed in tables of \textbf{Data Set 1} of Subsection \ref{sub52}. In view of \eqref{error1}, \eqref{error2}, and Theorem \ref{thm.main2}, the optimal convergence rate for both the velocity and pressure  is 2.0. From the five tables ($\kappa=0.1-0.5$) in \textbf{Data Set 1}, it is clear that  the optimal convergence rates for both variables are obtained on meshes when $\kappa=0.1$ and $\kappa=0.2$. For $\kappa\geq 0.3$, we do not have the optimal convergence rates even on graded meshes. In particular, on quasi-uniform meshes ($\kappa=0.5$), the rate is down to 0.93, which is far smaller than the best possible rate. This verifies our theoretical prediction. Namely, the optimal range of $\kappa$ is  $0<\kappa<0.212$ to achieve the optimal rate of convergence on $\Omega_1$.

Our second set of tests are for another domain $\Omega_2$ (the second domain in Figure \ref{fig1}), which are designed to justify our method for solutions with singularities on the $z$-axis. The interior angle at the vertex $Q$ of $\Omega_2$ is $0.75\pi$  and other interior angles $\leq 0.75\pi$. Based on the calculation on the eigenvalues of the corresponding operator pencil \cite{book} and \eqref{eta}, for the vertex $Q$, the parameter $\eta\approx 0.711+0.5=1.211$. In addition, for other vertices of the domain, we have $\eta>2$. Therefore,  by Theorem \ref{thm.main2}, we need to use   graded mesh near $Q$ with the parameter $\kappa<2^{-2/\eta}\approx0.318$ to approximate the singular solution at the optimal rate.

The numerical results for the second domain $\Omega_2$ are summarized in \textbf{Data Set 2} of Subsection \ref{sub52}. As in our first tests for $\Omega_1$, we clearly see the improvements on the convergence rates by using appropriate graded meshes. \textbf{Data Set 2}  shows that the P2-P1 Taylor-Hood approximations converge in the optimal rate on graded meshes with $\kappa\leq 0.3$ and the rates are slowing down for $\kappa\geq 0.4$, which convincingly supports our estimates in Theorem \ref{thm.main2}. Namely, the optimal range for the grading ratio is $0<\kappa<0.318$.

\subsection{Numerical outcomes} \label{sub52} We here collect the data from our numerical simulations for the P2-P1 Taylor-Hood approximation of the axisymmetric problem on both domains $\Omega_1$ and $\Omega_2$. The convergence rate for the velocity on the $j$th level is computed by 
\begin{eqnarray}\label{error1}
{\rm{rate}}_\bu=\log_2(\frac{\|\bu_{j-1}-\bu_{j-2}\|_{H^1_{-}(\Omega)\times H^1_+(\Omega)}}{\|\bu_{j}-\bu_{j-1}\|_{H^1_{-}(\Omega)\times H^1_+(\Omega)}}),
\end{eqnarray}
where $\bu_j$ is the numerical velocity on the $j$th level of the triangulation. The convergence rate for the pressure on the $j$th level is computed by 
\begin{eqnarray}\label{error2}
{\rm{rate}}_p=\log_2(\frac{\|p_{j-1}-p_{j-2}\|_{L^2_1(\Omega)}}{\|p_{j}-p_{j-1}\|_{L^2_1(\Omega)}}),
\end{eqnarray}
where $p_j$ is the numerical pressure on the $j$th level of the triangulation. These rates are good approximations of the asymptotic convergence rates given in Theorem \ref{thm.main2} in case  the exact solution is not known.\\\\
\textbf{Data Set 1.} Errors and convergence rates for the velocity and pressure on different levels of the graded mesh for $\Omega_1$:
\begin{center}
\begin{tabular}{|c||c|c|c|c|c|c|}
\hline
level  ($\kappa=0.1$) &  
            $\|\bu_j-\bu_{j-1}\|_{H^1_{-}(\Omega_1)\times H^1_+(\Omega_1)}$ & rate$_\bu$ &  $\|p_j-p_{j-1}\|_{L^2_1(\Omega_1)}$ & rate$_p$  \\
\hline
  $4$         &  0.54308907E-02 & x       &  0.61340596E-02 & x \\
  $5$    & 0.13864106E-02  & 1.970 & 0.15621365E-02 & 1.973 \\
  $6$   & 0.34352351E-03  & 2.013 & 0.38467436E-03 & 2.022 \\
  $7$   & 0.85001888E-04  & 2.015 & 0.94333002E-04  & 2.028 \\
  $8$   & 0.21124578E-04 & 2.009  &  0.23369704E-04 & 2.013 \\
\hline \end{tabular}\\\vspace{0.3cm}

\begin{tabular}{|c||c|c|c|c|c|c|}
\hline
level  ($\kappa=0.2$) &  
            $\|\bu_j-\bu_{j-1}\|_{H^1_{-}(\Omega_1)\times H^1_+(\Omega_1)}$ & rate$_\bu$ &  $\|p_j-p_{j-1}\|_{L^2_1(\Omega_1)}$ & rate$_p$  \\
\hline
  $4$  &   0.47745520E-02  & x        &  0.52875473E-02 & x \\
  $5$  &  0.12233139E-02  & 1.965 & 0.13410309E-02 & 1.979 \\
  $6$  &   0.30444567E-03  & 2.007 & 0.33060560E-03 & 2.020 \\
  $7$  &   0.75722130E-04  & 2.007 &  0.81497356E-04  & 2.020 \\
  $8$  &    0.18896191E-04 & 2.003 & 0.20268038E-04 & 2.008 \\
\hline
\end{tabular}\vspace{0.3cm}

\begin{tabular}{|c||c|c|c|c|c|c|}
\hline
level  ($\kappa=0.3$) &  
            $\|\bu_j-\bu_{j-1}\|_{H^1_{-}(\Omega_1)\times H^1_+(\Omega_1)}$ & rate$_\bu$ &  $\|p_j-p_{j-1}\|_{L^2_1(\Omega_1)}$ & rate$_p$  \\
\hline
  $4$        & 0.44173990E-02  & x        &  0.47079809E-02 & x \\
  $5$   & 0.11510414E-02  & 1.940 & 0.12027935E-02 & 1.969 \\
  $6$   & 0.29441776E-03  & 1.967 & 0.30352941E-03 & 1.987 \\
  $7$  & 0.76396705E-04  & 1.946 & 0.77859399E-04  & 1.963\\
  $8$    & 0.20346886E-04 & 1.909  & 0.20544175E-04 & 1.922 \\
\hline
\end{tabular}\\\vspace{0.3cm}

\begin{tabular}{|c||c|c|c|c|c|c|}
\hline
level  ($\kappa=0.4$) &  
            $\|\bu_j-\bu_{j-1}\|_{H^1_{-}(\Omega_1)\times H^1_+(\Omega_1)}$ & rate$_\bu$ &  $\|p_j-p_{j-1}\|_{L^2_1(\Omega_1)}$ & rate$_p$  \\
\hline
  $4$  &   0.43660475E-02  & x        &  0.44774198E-02 & x \\
  $5$  &   0.12281546E-02  & 1.830 & 0.12304147E-02 & 1.864 \\
  $6$  &   0.37522157E-03  & 1.711 & 0.37294953E-03 & 1.722 \\
  $7$  &  0.13214118E-03  & 1.506 & 0.13115375E-03  & 1.508 \\
  $8$  &  0.52227900E-04  & 1.339 & 0.51859766E-04 & 1.339 \\
\hline
\end{tabular}\\\vspace{0.3cm}

\begin{tabular}{|c||c|c|c|c|c|c|}
\hline
level  ($\kappa=0.5$) &  
            $\|\bu_j-\bu_{j-1}\|_{H^1_{-}(\Omega_1)\times H^1_+(\Omega_1)}$ & rate$_\bu$ &  $\|p_j-p_{j-1}\|_{L^2_1(\Omega_1)}$ & rate$_p$  \\
\hline
  $4$  &    0.47263869E-02  & x        &  0.47219944E-02 & x \\
  $5$  &     0.16133776E-02 & 1.551 & 0.15797647E-02 & 1.580 \\
  $6$  &    0.69734222E-03 & 1.210 & 0.68000773E-03 & 1.216 \\
  $7$  &   0.34933809E-03  & 0.997 & 0.33991358E-03  & 1.000 \\
  $8$  &   0.18340080E-03 & 0.930  & 0.17817426E-03 & 0.932 \\
\hline
\end{tabular}
\end{center}
\hspace{0.3cm}
\begin{center}
\includegraphics[scale=0.25]{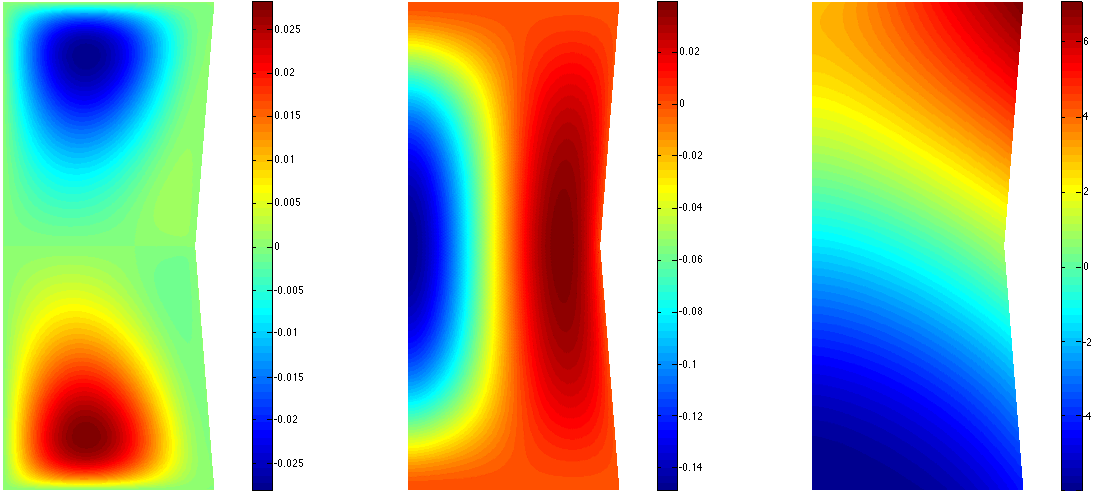}\\
{Numerical solutions on $\Omega_1$:  the radial component of the velocity (left);  the axial component of the velocity (center); the pressure (right).}
\end{center}
\vspace{0.3cm}
\textbf{Data Set 2.} Errors and convergence rates for the velocity and pressure on different levels of the graded mesh for $\Omega_2$:
\begin{center}
\begin{tabular}{|c||c|c|c|c|c|c|}
\hline
level ($\kappa=0.1$) &  
            $\|\bu_j-\bu_{j-1}\|_{H^1_{-}(\Omega_2)\times H^1_+(\Omega_2)}$ & rate$_\bu$ &  $\|p_j-p_{j-1}\|_{L^2_1(\Omega_2)}$ & rate$_p$ \\
\hline
  $4$  &    0.11515746E-01 & x &  0.15061640E-01 & x \\
  $5$  &    0.31781579E-02   & 1.858  & 0.41813381E-02 & 1.849 \\
  $6$  &    0.85557256E-03  & 1.893 &  0.11418853E-02 & 1.873 \\
  $7$  &   0.22457309E-03 & 1.930 & 0.30406439E-03 & 1.909 \\
  $8$  &    0.57563869E-04 & 1.964 & 0.78724222E-04 & 1.950 \\
\hline
\end{tabular}\\\vspace{0.3cm}

\begin{tabular}{|c||c|c|c|c|c|c|}
\hline
level ($\kappa=0.2$) &  
            $\|\bu_j-\bu_{j-1}\|_{H^1_{-}(\Omega_2)\times H^1_+(\Omega_2)}$ & rate$_\bu$ &  $\|p_j-p_{j-1}\|_{L^2_1(\Omega_2)}$ & rate$_p$ \\
\hline
  $4$  &    0.10464944E-01   & x        & 0.13784984E-01 & x \\
  $5$  &   0.28469055E-02   & 1.878 & 0.37366516E-02 & 1.883 \\
  $6$  &   0.76118324E-03    & 1.903 & 0.10031844E-02 & 1.897 \\
  $7$  &    0.19948945E-03   & 1.932 & 0.26373350E-03 & 1.927 \\
  $8$  &    0.51149045E-04   & 1.964 & 0.67587285E-04 & 1.964 \\
\hline
\end{tabular}\\\vspace{0.3cm}

\begin{tabular}{|c||c|c|c|c|c|c|}
\hline
level ($\kappa=0.3$) &  
            $\|\bu_j-\bu_{j-1}\|_{H^1_{-}(\Omega_2)\times H^1_+(\Omega_2)}$ & rate$_\bu$ &  $\|p_j-p_{j-1}\|_{L^2_1(\Omega_2)}$ & rate$_p$ \\
\hline
  $4$  &  0.98407840E-02 & x        &   0.12914495E-01 & x \\
  $5$  &   0.26674362E-02 & 1.883 &  0.34435199E-02 & 1.907 \\
  $6$  &   0.71166545E-03  & 1.906 & 0.91048474E-03 & 1.919 \\
  $7$  &   0.18647577E-03  & 1.932 & 0.23637606E-03 & 1.946  \\
  $8$  &   0.47900514E-04  & 1.961 & 0.60046266E-04 & 1.977 \\
\hline
\end{tabular}\\\vspace{0.3cm}

\begin{tabular}{|c||c|c|c|c|c|c|}
\hline
level ($\kappa=0.4$) &  
            $\|\bu_j-\bu_{j-1}\|_{H^1_{-}(\Omega_2)\times H^1_+(\Omega_2)}$ & rate$_\bu$ &  $\|p_j-p_{j-1}\|_{L^2_1(\Omega_2)}$ & rate$_p$ \\
\hline
  $4$  &  0.95957034E-02 & x &   0.12356731E-01 & x \\
  $5$  &   0.26421705E-02 & 1.861 &  0.32744105E-02 & 1.916 \\
  $6$  &   0.72574943E-03 & 1.864 &  0.86734770E-03 & 1.917 \\
  $7$  &   0.19991811E-03 & 1.860  &  0.22897740E-03 & 1.921 \\
  $8$  &   0.55620873E-04 & 1.846 &  0.60571541E-04 & 1.919 \\
\hline
\end{tabular}\\\vspace{0.3cm}

\begin{tabular}{|c||c|c|c|c|c|c|}
\hline
level ($\kappa=0.5$) &  
            $\|\bu_j-\bu_{j-1}\|_{H^1_{-}(\Omega_2)\times H^1_+(\Omega_2)}$ & rate$_\bu$ &  $\|p_j-p_{j-1}\|_{L^2_1(\Omega_2)}$ & rate$_p$ \\
\hline
  $4$  &  0.10091361E-01 & x &   0.12358326E-01 & x \\
  $5$  &   0.31028469E-02 & 1.702  &  0.34365922E-02 & 1.846 \\
  $6$  &   0.10525854E-02 & 1.560  &  0.10381522E-02 & 1.727 \\
  $7$  &   0.39684959E-03 & 1.407  &  0.35441706E-03 & 1.551 \\
  $8$  &   0.16108656E-03 & 1.301  &  0.13597125E-03 & 1.382 \\
\hline
\end{tabular}
\end{center}

\vspace{0.3cm}
\begin{center}
\includegraphics[scale=0.27]{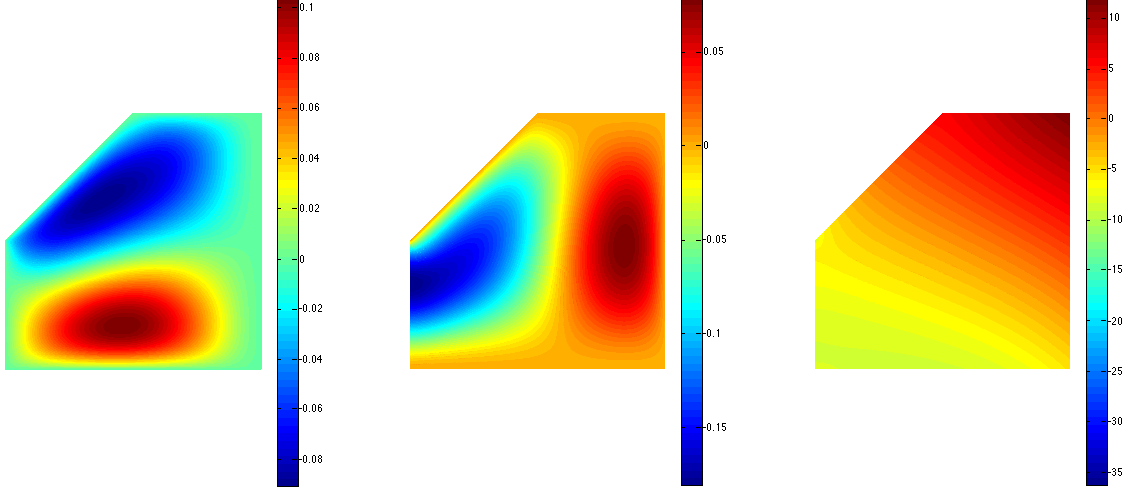}\\
{Numerical solutions on $\Omega_2$:  the radial component of the velocity (left);  the axial component of the velocity (center); the pressure (right).}
\end{center}

\bibliography{fem}
\bibliographystyle{plain}

\end{document}